\documentclass[11pt]{amsart}
\usepackage{amsmath, graphicx, amsfonts, amssymb}
\usepackage{amsfonts, mathrsfs, color, amsthm, amsbsy}
\usepackage{mathtools}
\usepackage{dsfont} 
\usepackage{adjustbox}
\usepackage{enumerate}
\usepackage[normalem]{ulem}
\usepackage{amsrefs}
\usepackage{comment}
\usepackage{MnSymbol}
\DeclareMathAlphabet{\mathpzc}{OT1}{pzc}{m}{it}
\usepackage[utf8]{inputenc}

\usepackage{geometry} 
\usepackage{parskip}
\newtheorem{thmx}{Theorem}
\usepackage{hyperref}
\renewcommand\H{\mathcal{H}}
\newcommand{\weakly}{\rightharpoonup}
\newcommand\BC{\mathbf{C}}
\newcommand\del{\partial}
\newcommand\dist{\textnormal{dist}}
\renewcommand{\index}{\textnormal{index}}
\newcommand\inj{\textnormal{inj}}

\usepackage[symbol]{footmisc}

\usepackage{tikz}
\usetikzlibrary{positioning,patterns,calc}

\newcommand{\duline}[1]{{\bgroup \markoverwith{{\bgroup \markoverwith{\rule[-1.2pt]{0.1pt}{0.4pt}}\ULon {\rule[-2.8pt]{1pt}{0.4pt}}}}\ULon {#1}}} 
\newcommand{\fduline}[1]{{\bgroup \markoverwith{{\bgroup \markoverwith{\rule[-1.2pt]{0.4pt}{0.4pt}}\ULon {\rule[-2.8pt]{2pt}{0.4pt}}}}\ULon {#1}}} 

\newcommand{\real}{\mathbb{R}} 
\newcommand\R{\mathbb{R}} 
\newcommand{\integer}{\mathbb{Z}} 
\newcommand{\haus}{\mathcal{H}} 
\newcommand{\eball}{B}

\newcommand{\csvf}{C^1_c}

\newcommand{\divergence}{\textup{div}} 
\newcommand{\graph}{\textup{graph}} 
\newcommand{\ivarifolds}{\mathcal{IV}} 

\newcommand{\setv}{\mathbf{v}}

\newcommand{\supp}{\textnormal{spt}}
\newcommand{\spt}{\textnormal{spt}}

\newcommand{\Index}{\textup{index}} 

\newcommand{\restrictv}{\mathbin{\hspace{0.1em}\vrule height 1.3ex depth 0pt width 0.13ex\vrule height 0.13ex depth 0pt width 1.0ex}}

\newcommand{\tangentcones}{\textup{VarTan}} 
\newcommand{\folding}{\mathfrak{f}} 
\newcommand{\reg}{\textup{reg}} 
\newcommand{\sing}{\textup{sing}} 
\newcommand{\salpha}{\mathscr{S}} 
\newcommand{\regscale}{\pmb{r}} 
\newcommand{\badreg}{\mathcal{B}} 
\newcommand{\strata}{\mathcal{S}} 
\newcommand{\kcone}{\mathscr{C}} 

\newtheorem{theorem}{Theorem}[section]
\newtheorem{lemma}[theorem]{Lemma}

\newtheorem{corollary}[theorem]{Corollary}

\newtheorem{definition}[theorem]{Definition}
\newtheorem*{theorem*}{Theorem}

\makeatletter
\@addtoreset{claim}{theorem}
\makeatother

\usepackage{etoolbox}
\makeatletter
\patchcmd{\@maketitle}
  {\ifx\@empty\@dedicatory}
  {\ifx\@empty\@date \else {\vskip3ex \centering\footnotesize\@date\par\vskip1ex}\fi
   \ifx\@empty\@dedicatory}
  {}{}
\patchcmd{\@adminfootnotes}
  {\ifx\@empty\@date\else \@footnotetext{\@setdate}\fi}
  {}{}{}
\makeatother

\title[Quantitative Estimates on the Singular Set of Minimal Hypersurfaces]{Quantitative Estimates on the Singular Set of Minimal Hypersurfaces with Bounded Index}
\author{Nicolau S. Aiex, Sean McCurdy, and Paul Minter}
\date{\today}
\address{Department of Mathematics, National Taiwan Normal University, Taipei, Taiwan}
\email{nsarquis@math.ntnu.edu.tw}
\address{Department of Mathematics, National Taiwan Normal University, Taipei, Taiwan}
\email{smccurdy@ntnu.edu.tw}
\address{Institute for Advanced Study (Fuld Hall) and Princeton University (Fine Hall), Princeton, New Jersey, USA, 08544, USA}
\email{pm6978@princeton.edu\textnormal{,} pminter@ias.edu}

\begin{document}

\begin{abstract}
We prove local measure bounds on the tubular neighbourhood of the singular set of codimension one stationary integral $n$-varifolds $V$ in Riemannian manifolds which have both: (i) finite index on their smoothly embedded part; and (ii) $\H^{n-1}$-null singular set. A direct consequence of such a bound is a bound on the upper Minkowski content of the singular set of such a varifold in terms of its total mass and its index. Such a result improves on known bounds, namely the corresponding bound on the $\H^{n-7}$-measure of the singular set established by A.~Song (\cite{antoine-song}), as well as the same bounds established by A.~Naber and D.~Valtorta (\cite{naber-valtorta}) for codimension one area minimising currents. Our results also provide more structural information on the singular set for codimension one integral varifolds with finite index (on the regular part) and no classical singularities established by N.~Wickramasekera (\cite{wickstable}). 
\end{abstract}

\maketitle

\section{Introduction}

The aim of the present paper is to prove quantitative estimates on the size of the singular set for a large class of codimension one stationary integral varifolds which have a smoothly embedded part of finite index (see Section \ref{sec:prelim} for precise definitions). In particular, we prove measure bounds on the size of the tubular neighbourhood of the singular set (and thus bound the upper Minkowski content of the singular set) in terms of the total (varifold) mass and the index. More precisely, our main result in the Euclidean setting can be stated as follows:

\begin{thmx}\label{thm:A}
    Let $\Lambda>0$ and $n\geq 8$. Suppose $V$ is a stationary integral $n$-varifold in $B^{n+1}_2(0)$ such that: \textnormal{(i)} $\|V\|(B_2^{n+1}(0))\leq \Lambda$; \textnormal{(ii)} its smoothly embedded part has finite index, i.e., $\textnormal{index}(\reg(V))<\infty$; and \textnormal{(iii)} $\mathcal{H}^{n-1}(\sing(V))=0$. Then, $\sing(V)$ is countably $(n-7)$-rectifiable, and moreover we have, for any $0<r\leq 1/2$,
    $$\mathcal{H}^{n+1}\left(\eball_{r/8}(\sing(V))\cap \eball_{1/2}\right) \leq C_0\left(1+\textnormal{index}(\reg(V))\right)r^8;$$
    $$\|V\|\left(\eball_{r/8}(\sing(V))\cap \eball_{1/2}\right)\leq C_0(1+\textnormal{index}(\reg(V)))r^7.$$
    In particular, the upper Minkowski content of $\sing(V)$ obeys
    $$\mathcal{M}^{* n-7}\left(\sing(V)\cap \eball_{1/2}\right) \leq C_0(1+\textnormal{index}(\reg(V)))$$
    which in turn implies that $\mathcal{H}^{n-7}\left(\sing(V)\cap\eball_{1/2}\right) \leq C_0(1+\textnormal{index}(\reg(V)))$; here, $C_0 = C_0(n,\Lambda)\in (0,\infty)$.
\end{thmx}

When our varifolds instead are defined on an ambient smooth Riemannian manifold, our main result is:

\begin{thmx}\label{thm:B1}
    Let $\Lambda>0$, $n\geq 8$, $K>0$. Let $(N^{n+1},g)$ be a smooth Riemannian manifold with $0\in N$ and obeying $\left|\left.\textnormal{sec}\right|_{B^N_2(0)}\right|\leq K$ and $\left.\textnormal{inj}\right|_{B^N_2(0)}\geq K^{-1}$, where here $\textnormal{sec}$ is the sectional curvature of $N$ and $\textnormal{inj}$ is the injectivity radius. Suppose that $V$ is a stationary integral $n$-varifold in $B^N_2(0)$ which obeys: \textnormal{(i)} $\|V\|(B^{N}_2(0))\leq \Lambda$; \textnormal{(ii)} the smoothly embedded part of $V$ has finite index; and \textnormal{(iii)} $\H^{n-1}(\sing(V)) = 0$. Then, $\sing(V)$ is countably $(n-7)$-rectifiable, and moreover we have, for any $0<r\leq 1/2$,
    $$\H^{n+1}(B^N_{r/8}(\sing(V))\cap B^N_{1/2}(0))\leq C_0\left(1+\index(\reg(V))\right)r^8;$$
    $$\|V\|(B^N_{r/8}(\sing(V))\cap B^N_{1/2}(0))\leq C_0\left(1+\index(\reg(V))\right)r^7.$$
    In particular, we have $\mathcal{M}^{*n-7}(\sing(V)) \leq C_0\left(1+\index(\reg(V))\right)$; here, $C_0 = C_0(n,\Lambda,K)\in (0,\infty)$ and $B^N_r(A)$ denotes the $r$-neighbourhood in $N$ of the a subset $A\subset N$.
\end{thmx}

We remark that from Theorem \ref{thm:B1}, a simple covering argument gives a global result in closed Riemannian manifolds, as there is a uniform bound on the sectional curvature and a uniform lower bound on the injectivity radius (depending on $(N^{n+1},g)$):

\begin{thmx}\label{thm:B}
Let $\Lambda> 0$, $n\geq 8$, and let $(N^{n+1},g)$ be a closed smooth Riemannaian manifold. Suppose $V$ is a stationary integral $n$-varifold in $N$ obeying: \textnormal{(i)} $\|V\|(N)\leq \Lambda$; \textnormal{(ii)} the smoothly embedded part of $V$ has finite index; and \textnormal{(iii)} $\mathcal{H}^{n-1}(\sing(V)) = 0$. Then, $\sing(V)$ is countably $(n-7)$-rectifiable, and moreover we have, for any $0<r\leq 1/2$,
$$\mathcal{H}^{n+1}\left(B^N_{r/8}(\sing(V))\right)\leq C_0(1+\textnormal{index}(\reg(V)))r^8;$$
$$\|V\|\left(B^N_{r/8}(\sing(V))\right) \leq C_0(1+\textnormal{index}(\reg(V)))r^7.$$
In particular, we have $\mathcal{M}^{*n-7}(\sing(V))\leq C_0(1+\textnormal{index}(\reg(V)))$; here, $C_0 = C_0(\Lambda,N,g)$.
\end{thmx}

In the special case where $V$ is the varifold associated to a codimension one area minimising current $T$ (which in particular obey $\textnormal{index}(\reg(V)) = 0$), we note that the corresponding results have already been established in \cite{naber-valtorta}*{Theorem 1.6}.

Under the same assumptions, A.~Song (\cite{antoine-song}) established that the $\mathcal{H}^{n-7}$-measure of the singular set obeys, for each $n\geq 7$,
$$\mathcal{H}^{n-7}(\sing(V)\cap \eball_{1/2}) \leq C^{\text{AS}}_0(1+\textnormal{index}(\reg(V)))^{7/n}$$
where $C^{\text{AS}}_0 = C^{\text{AS}}_0(\Lambda, N,g)$. In particular, in the special case when $n=7$ (and thus the singular set is $0$-dimensional, and in fact by \cite{wickstable} consists only of isolated points), this already implies measure bounds on the tubular neighbourhood of $\sing(V)$; indeed, it is for this reason that we are only interested in the case $n\geq 8$ (although our argument only works for $n\geq 8$, much in the same way as how the corresponding argument in \cite{antoine-song} for $n\geq 8$ also cannot be used to prove the $n=7$ case). However, when $n\geq 8$, the $\H^{n-7}$-measure bound is significantly weaker than control on the measure of the entire tubular neighbourhood. We also remark that, as the constants $C_0,C_0^{\text{AS}}$ depend on the total varifold mass, it is unclear which bound on the $\mathcal{H}^{n-7}$-measure of the singular set is optimal, as there are conjectured relationships between the area of a stationary integral varifold and the index of its embedded part. It would be an interesting question to determine the dependence of the constant $C_0$ on the mass $\|V\|(B^{n+1}_2(0))$. 

As a further remark, since $\dim_{\mathcal{H}}(\sing(V))\leq n-7$, we know that $\reg(V)$ is a connected smooth submanifold, and as such by the constancy theorem occurs with some constant (integer) multiplicity.  Thus, without loss of generality we can assume the multiplicity is 1 when proving the first estimate of our theorems as this does not change the value of the index. Under this assumption, $\|V\|(B^{n+1}_2(0)) = \mathcal{H}^{n}(\reg(V))$, and if we set $M:= \reg(V)$, we can work directly with $M$, viewing $\sing(V)\equiv \sing(M):= \overline{M}\backslash M$. Furthermore, this tells us that the constant $C_0$ in our main theorems, for the first inequality at least, only depends on $\H^n(\reg(V)\cap B^{n+1}_2(0))$, and not on $\Lambda$.

It should be noted that, whilst the results of \cite{antoine-song} are stated under the assumption that the singular set a priori obeys the stronger assumption $\dim_{\mathcal{H}}(\sing(V))\leq n-7$ (i.e., for each $\gamma>0$, $\mathcal{H}^{n-7+\gamma}(\sing(V)) = 0$), under the assumption that the smoothly embedded part of $V$ has finite index, the assumption that $\mathcal{H}^{n-1}(\sing(V)) = 0$ \textit{actually implies} $\dim_{\mathcal{H}}(\sing(V))\leq n-7$. This is due to the fact that when the index is finite, the smoothly embedded part of the varifold is locally stable, and thus one may apply the regularity theory for stable codimension one stationary integral varifolds established by N.~Wickramasekera (\cite{wickstable}). Consequently, $\sing(V) = \mathcal{S}_{n-7}$ is equal to the $(n-7)^{\text{th}}$ strata, and thus is countably $(n-7)$-rectifiable by \cite{naber-valtorta}. Moreover, it should be noted that \cite{wickstable} also gives that one may replace in our results the assumption that $\mathcal{H}^{n-1}(\sing(V)) = 0$ with the equivalent condition that $V$ does not contain any so-called \textit{classical singularities}, i.e. singularities locally about which $V$ can be written as a sum of at least 3 $C^{1,\alpha}$ submanifolds-with-boundary, for some $\alpha\in (0,1)$, which all have the same common boundary.


\textbf{
Acknowledgements:} The first author would like to thank Professor Ulrich Menne for suggesting the problem and putting this collaboration together.
NSA 
was supported through the grant with No. MOST 108-2115-M-003-016-MY3 by 
the National Science and Technology Council.


\section{Preliminaries}\label{sec:prelim}

Throughout this article $n$ will always denote a positive integer unless its range is specified.
Given $r>0$, we denote by $\eball^{n+1}_r(x)$ the Euclidean ball of radius $r$ centered at $x\in\real^{n+1}$, and as a shorthand we write $\eball^{n+1}_r(0)\equiv \eball^{n+1}_r$.
If $A\subset\real^{n+1}$ is any set and $x\in\real^{n+1}$ then we write $d(x,A):=\inf\{d(x,a):a\in A\}$, $\eball_r(A):=\{x\in\real^{n+1}:d(x,A)<r\}$ and given $B\subset\real^{n+1}$ another subset, we define their distance as $d(A,B):=\inf\{d(a,b):a\in A,b\in B\}$.

Given an open set $U\subset\real^{n+1}$ and a positive integer $k$ we will denote by $\ivarifolds_k(U)$ the space of integral $k$-varifolds in $\real^{n+1}$ with support in $U$.
We endow $\ivarifolds_k(U)$ with the weak topology of measure-theoretic convergence, which is induced by the corresponding Fr\' echet structure, for which we write $\mathbf{d}$ for the corresponding metric which induces the topology.
Given $V\in\ivarifolds_k(U)$ and $x\in\supp(V)$ we denote by $\Theta_V(x)$ its density at $x$, which is a positive integer $\|V\|$-almost everywhere.
If $R\subset U$ is a $\haus^k$-rectifiable set and $\theta:U\rightarrow\integer_{\geq 0}$ is a $\haus^k$-measurable function which is locally integrable on $U$, we denote the induced integral $k$-varifold of density $\theta$ by $\setv(R,\theta)$; in the special case when $R$ is the graph of a $C^2$ function $u$ and the multiplicity $\theta$ is identically 1, we write the shorthand $\setv(R,\theta)\equiv \setv(u)$.
We will always assume that $1\leq k\leq n$ when providing general definitions since the extreme cases $k=0$ and $k=n+1$ are often trivial and irrelevant to usual applications.

For $x_0\in \R^{n+1}$ and $r>0$, we define the homothetic rescaling by $r$ about $x_0$ to be the function $\eta_{x_0,r}:\R^{n+1}\to \R^{n+1}$ given by $\eta_{x_0,r}(x):= r^{-1}(x-x_0)$. Given $V\in \ivarifolds_k(U)$, $x\in\spt\|V\|$, and $r>0$, we write $\tangentcones_x(V)\subset\ivarifolds_k(\R^{n+1})$ for the set of tangent cones to $V$ at $x$, i.e. accumulation points of (convergent subsequences of) $(\eta_{x,r})_\#V$ as $r\downarrow 0$.

Given $V\in\ivarifolds_k(U)$ we define its \textit{regular part} as
\begin{equation*}
\begin{aligned}
\reg(V) := \{ & x\in\supp\|V\|: \supp\|V\|\cap\eball_\rho(x) \text{ is a properly smoothly embedded }\\
             &\hspace{17em} k\text{-submanifold of }U \text{ for some }\rho>0 \}
\end{aligned}
\end{equation*}
and its \textit{singular part} to be $\sing(V):=\supp\|V\|\backslash\reg(V)$.
We note that, according to the definition above, even if a varifold $V$ is given by a smooth immersed submanifold, its points of self-intersection are in $\sing(V)$.

Let $\csvf(U,\real^{n+1})$ be the space of compactly supported vector fields of class $C^1$ in $U$.
If $V\in\ivarifolds_k(U)$, $X\in\csvf(U,\real^{n+1})$, and $\phi^X_t$ is the $C^2$ flow generated by $X$ defined for $t\in(-\varepsilon,\varepsilon)$, we denote the first and second variations of $V$ in the direction of $X$ by $\delta V(X)=\left.\frac{d}{dt}\right|_{t=0}\|(\phi^X_t)_\# V\|$ and $\delta^2 V(X,X)=\left.\frac{d^2}{dt^2}\right|_{t=0}\|(\phi^X_t)_\# V\|$ respectively.
We say that $V$ is \emph{stationary} in $U$ if $\delta V(X)=0$ for all $X\in\csvf(U,\real^{n+1})$ and of \emph{bounded first variation} if there exists $H>0$ such that $\delta V(X)\leq H\int_U|X|d\|V\|$ for all $X\in\csvf(U,\real^{n+1})$. A direct computation shows (\cite{simon-gmt}):
$$\delta V(X) = \int_{U\times G(k,n+1)}\textnormal{div}_S(X_x)\ dV(x,S)$$
and
\begin{align*}
    \delta^2V(X,X) = \int_{U\times G(k,n+1)}&\left\{\divergence_S(Y_x) + \left(\divergence_S(X_x)\right)^2 + \sum^k_{i=1}\left|(\nabla_{\tau_i}X_x)^{\perp_S}\right|^2\right.\\
    & \hspace{3em} \left. - \sum^k_{i,j=1}(\tau_j\cdot \nabla_{\tau_i}X_x)(\tau_i\cdot \nabla_{\tau_j}X_x)\right\}\ d V(x,S)
\end{align*}

where $Y_x:= \left.\frac{d^2}{dt^2}\right|_{t=0}\phi^X_t(x)$, $\nabla$ is the Euclidean connection, and $\{\tau_1,\dotsc,\tau_k\}$ is a choice of orthonormal basis for the subspace $S\in G(k,n+1)$, with $\perp_S$ denoting the corresponding orthogonal projection onto the orthogonal complement of $S$.

\begin{definition}\label{definition bounded index}
Let $\Omega\subset U\subset\real^{n+1}$ be open sets,  $V\in\ivarifolds_k(U)$ a stationary integral varifold and $I\in\integer_{\geq 0}$ a non-negative integer.
We say that the regular part of $V$ has index bounded by $I$ in $\Omega$ if and only if for all subspaces $P\subset\csvf(\Omega\backslash\sing(V),\R^{n+1})$ of dimension $I+1$ there exists $X\in P$ such that $X\neq 0$ and $\delta^2 V (X,X)\geq 0$.

When the regular part of $V$ has bounded index in $\Omega$, we may define its corresponding index as
\begin{equation*}
\Index(\reg(V);\Omega):=\min\{I\in\integer_{\geq 0}: \reg(V)\text{ has index bounded by } I\text{ in }\Omega\}
\end{equation*}
and we write $\textnormal{index}(\reg(V)):= \textnormal{index}(\reg(V);U)$.
\end{definition}

The regular part of a varifold $V$ is said to be \emph{stable} in an open subset $\Omega\subset U$ if $\Index(\reg(V);\Omega)=0$.
Otherwise we say that the regular part of $V$ is \textit{unstable} in $\Omega$.
The following related notion of \emph{folding number} was introduced in \cite{antoine-song}: 

\begin{definition}\label{defn:folding}
Let $U\subset\real^{n+1}$ be an open subset and $V\in\ivarifolds_k(U)$ be a stationary integral varifold. We define the folding number $\folding(V)$ of $V$ as follows: if $V$ is stable in $U$, then $\folding(V):= 0$; otherwise, we define $\folding(V)$ to be the largest size of a (possibly infinite) collection of disjoint unstable open subsets in $U$.
\end{definition}

The folding number simply quantifies the number of disjoint unstable subsets in $V$. However, each such subset may have index strictly larger than $1$ without the possibility of breaking into further disjoint unstable subsets. The core property of varifolds with bounded index is the fact that they cannot be unstable in too many disjoint open subsets. This fact is quantified by the folding number in the following lemma.

\begin{lemma}\label{bounded index disjoint sets}
Let $U\subset\real^{n+1}$ be an open set, $I\in\integer_{\geq 0}$ be a non-negative integer, and $V\in\ivarifolds_l(U)$ be a stationary integral varifold with $\Index(\reg(V))\leq I$. If $\Omega_1,\ldots,\Omega_{I+1}\subset U$ is a collection of $I+1$ open sets such that $\Index(\reg(V);\Omega_j)\geq 1$ for each $j=1,\ldots,I+1$ then there exist $j\neq j'$ such that $\Omega_{j}\cap\Omega_{j'}\neq\emptyset$. In particular $\folding(V)\leq I$.
\end{lemma}

\begin{proof}
Suppose the lemma is false. Then, we may find a pairwise disjoint collection of open subsets $\Omega_1,\dotsc,\Omega_{I+1}\subset U$ for which $\textnormal{index}(\reg(V);\Omega_j)\geq 1$ for each $j=1,\dotsc,I+1$. By definition, for each such $j$ we can then find $X_j\in C^1_c(\Omega_j\backslash\sing(V);\R^{n+1})$ such that $\delta^2 V(X_j,X_j)<0$. By extension by $0$ outside $\spt(X_j)$, we can view each $X_j$ as an element of $C^1_c(U\backslash\sing(V);\R^{n+1})$. Moreover, we know for $i\neq j$ that $\spt(X_i)\cap \spt(X_j) \subset \Omega_i\cap \Omega_j = \emptyset$, and thus if we set $P:= \textnormal{span}\{X_1,\dotsc,X_{I+1}\}$, $P$ is a subspace of dimension $I+1$ in $C^1_c(U\backslash\sing(V);\R^{n+1})$, and $\delta^2V(X,X)<0$ for all $X\in P$. This then implies that $\Index(\reg(V))\geq I+1$, which contradicts our assumption that $\Index(\reg(V))\leq I$. Thus the lemma must hold. 
\end{proof}

Another useful quantity somewhat dual to the folding number is a useful radius quantity known as the \textit{stability radius}, defined analogously to that in \cite{antoine-song}:

\begin{definition}\label{defn:stab-radius}
Let $U\subset\R^{n+1}$ be an open subset and $V\in \ivarifolds_k(U)$ be a stationary integral varifold. Then the stability radius is the function $s_V:U\to [0,\infty]$ defined by:
$$s_V(x):= \sup\{r\geq 0:\textnormal{index}(\reg(V);\eball_r(x)\cap U) = 0\}.$$
\end{definition}

In this definition, we take $\textnormal{index}(\reg(V);\eball_r(x)) =0$ when $\spt\|V\|\cap\eball_r(x) = \emptyset$. Clearly, if $s_V(x)<\infty$ then $\reg(V)$ is unstable on $\eball_{\lambda s_V(x)}(x)$ for all $\lambda>1$. This immediately implies that if $s_V(x)<\infty$ for some $x\in U$, then $s_V(y)<\infty$ for all $y\in U$, and in particular $s_V(y) \leq s_V(x) + |x-y|$; indeed, if this inequality were not true, then we could find $\lambda>1$ for which $\reg(V)$ is stable in $\eball_{\lambda(s_V(x)+|x-y|)}(y)$. But $\eball_{\tilde{\lambda}s_V(x)}(x)\subset \eball_{\lambda(s_V(x)+|x-y|)}(y)$ for $\tilde{\lambda}\in (1,\lambda)$ sufficiently close to $1$, and thus we would have $B_{\tilde{\lambda}s_V(x)}(x)$ is stable (as the index is non-decreasing with respect to set-inclusion), a contradiction to the definition of $s_V(x)$. In particular, we have a dichotomy that either $s_V\equiv \infty$ (i.e. $\reg(V)$ is stable in $U$) or $s_V$ is finite everywhere.

The above simple argument shows that the stability function is always Lipschitz with Lipschitz constant at most $1$ when $V$ is not stable in $U$; in particular, $s_V$ is continuous. 
\begin{lemma}\label{stability radius continuity}
    Let $U\subset\R^{n+1}$ be open and $V\in \ivarifolds_k(U)$ be a stationary integral varifold. Then, either $s_V\equiv \infty$, or we have for all $x,y\in U$, $|s_V(x)-s_V(y)|\leq |x-y|$.
\end{lemma}
\begin{proof}
    Suppose $s_V\not\equiv \infty$, and so $s_V<\infty$ pointwise. In the above discussion we saw that for any $x,y\in U$, we have $s_V(y) \leq s_V(x) + |x-y|$. Swapping $x$ and $y$ in this identity and combining gives $|s_V(x) - s_V(y)| \leq |x-y|$, as desired.
\end{proof}

The next fact we need is that stationary integral varifolds with bounded index on their regular part are in fact locally stable about every point. This is a well-known fact about regular points, as indeed a smooth minimal submanifold is in fact locally area-minimising (see \cite{federer-2}*{Section 4}) and hence locally stable.

\begin{lemma}\label{locally stable}
    Let $U\subset\R^{n+1}$ be open and let $V\in \ivarifolds_k(U)$ be a stationary integral varifold obeying $\Index(\reg(V))<\infty$. Then, for each $x\in \spt\|V\|$, there exists $\rho_x>0$ for which $\Index(\reg(V);\eball_{\rho_x}(x)) = 0$; in particular, $s_V(x)\geq \rho_x>0$.
\end{lemma}

\begin{proof}
    As discussed before the statement of the lemma, this result follows by \cite{federer-2}*{Section 4} whenever $x\in\reg(V)$. So let us assume that $x\in \sing(V)$. We first claim that there exists $\rho_x>0$ for which $$\Index(\reg(V);\eball_{\rho_x}(x)\backslash\{x\}) = 0.$$
    Indeed, suppose this were not true. Then, for each $\rho>0$, we would have $$\Index(\reg(V);\eball_{\rho}(x)\backslash\{x\})\geq 1.$$ Hence, we can find a (non-zero) vector field $X\in C^1_c(U\backslash\sing(V);\R^{n+1})$ for which $\delta^2V(X,X) <0$ and $\spt(X)\subset (B_\rho(x)\backslash\{x\})\cap \reg(V)$. In particular, there is a $\tau_\rho\in (0,\rho)$ for which $\spt(X)\subset B_\rho(x)\backslash B_{\tau_\rho}(x)$. Thus, we may find a sequence $\rho_i\downarrow 0$ with $\rho_{i+1}<\tau_{\rho_i}$ for all $i$ and, for each $i\geq 1$, $\Index(\reg(V);B_{\rho_i}(x)\backslash B_{\tau_{\rho_i}}(x))\geq 1$. But then as $(B_{\rho_i}(x)\backslash B_{\tau_{\rho_i}}(x))_{i=1}^\infty$ are pairwise disjoint open sets, this is a direct contradiction to the fact that $\Index(\reg(V))<\infty$ (the contradiction coming via Lemma \ref{bounded index disjoint sets}). Hence the claim holds.
    
    We now further claim that, with this radius $\rho_x$, $\Index(\reg(V),\eball_{\rho_x}(x)) = 0$. Indeed, as $x\in \sing(V)$, if this were not true then one may find a non-zero $X\in C^1_c(\eball_{\rho_x}(x)\backslash\sing(V);\R^{n+1})$ for which $\delta^2V(X,X)<0$. But as $x\in \sing(V)$, $\spt(X)\subset \eball_{\rho_x}(x)\backslash \{x\}$, and so there is a $\tau>0$ for which $\spt(X)\subset \eball_{\rho_x}(x)\backslash\eball_{\tau}(x)$, a direct contradiction to the above claim. This concludes the proof of the lemma.
\end{proof}

We now introduce the last general piece of terminology we need, namely the \textit{regularity scale} of a varifold (as in \cite{cheeger-naber2013}*{Section 5.3}).

Given a $k$-dimensional subspace $P\subset \R^{n+1}$, let $\{\eta_1,\dotsc,\eta_{n+1-k}\}$ be a basis of $P^\perp$. We can then represent functions $B_1\cap P\to P^\perp$ as functions $u:B_1\cap P\to \R^{n+1-k}$, with their graphs being given by $\graph(u):= \{p+\sum_{j}u^j(p)\eta_j: p\in B_1\cap P\}$, where $u = (u^1,\dotsc,u^{n+1-k})$.

Now consider $V\in \ivarifolds_k(U)$, $x\in \spt\|V\|$, $\rho>0$, and $q\in \mathbb{Z}_{\geq 0}$. We say that $V\restrictv \eball_\rho(x)$ is a \textit{union of $q$ $C^2$ functions} if there exist $k$-dimensional subspaces $P_1,\dotsc,P_q\subset\R^{n+1}$, and $C^2$ functions $(u_i)_{i=1}^{q}$ with $u_i:B_1\cap P_i\to \R^{n+1-k}$ such that
$$(\eta_{x,\rho})_\#V = \sum^q_{i=1} \mathbf{v}(u_i)\restrictv B_1.$$

\textbf{Note:} The functions $u_i$ need not be distinct, as they can, for example, coincide when $\reg(V)$ occurs with multiplicity $>1$. Moreover, $q$ is not constant: two surfaces intersecting transversely are locally the graph of $2$ distinct functions (with different domains) near any point of intersection, and locally the graph of a single function away from the intersection points.
\begin{definition}\label{regularity scale}
    Let $U\subset\R^{n+1}$ be open, $V\in \ivarifolds_k(U)$, $x\in \spt\|V\|$, and $Q\in\integer_{>0}$. First define
    $$\regscale^Q_{0,V}(x):= \sup\{\rho>0: \eball_\rho(x)\subset U\text{ and }V\restrictv\eball_\rho(x)\textit{ is a union of at most }Q\ C^2\text{ functions}\}$$
    where we set $\sup(\emptyset) := 0$. Now, if $x\in \spt\|V\|$ is such that $\regscale^Q_{0,V}(x) = \rho>0$, and $y\in \eball_\rho(x)\cap \spt\|V\|$, then there exists $q\leq Q$, $k$-dimensional subspaces $P_1,\dotsc,P_q\subset\R^{n+1}$, and $C^2$ functions $(u_i)_{i=1}^q$, $u_i:B_1\cap P_i\to \R^{n+1-k}$ (all depending on $y$) with $y\in \graph(u_i)$ for each $i=1,\dotsc,q$; let $q$ be maximal obeying this. We then define the norm of the second fundamental form of $V$ at $y$ as:
    $$|A_V(y)|:= \sum^q_{i=1}|A_{\graph(u_i)}(y)|.$$
    Then, the regularity scale of $V$ at $x\in \spt\|V\|$ is defined to be
    $$\regscale^Q_{V}(x):= \sup\left\{0<\rho\leq \regscale^Q_{0,V}(x): \sup_{\spt\|V\|\cap \eball_\rho(x)}\rho|A_V| \leq 1\right\}$$
    where once again, if this set is empty (i.e. $\regscale^Q_{0,V}(x) = 0$) we set $\regscale^Q_{V}(x):= 0$.
    
    We then write $\badreg_r^Q(V):= \{x\in\spt\|V\|:\regscale^Q_V(x)\leq r\}$ and $\badreg_r(V):= \{x\in \spt\|V\|:\regscale^Q_V(x)\leq r\ \text{for some }Q\}$.
\end{definition}

\subsection{Codimension one integral varifolds in $\real^{n+1}$}
In this section we give details of the relevant results and definitions for codimension one stationary integral varifolds that will be necessary for our work.

\begin{definition}\label{s-alpha condition}
    Let $n\in\mathbb{Z}_{\geq 2}$, $I\in \mathbb{Z}_{\geq 0}$, and $V\in \ivarifolds_n(\eball_2^{n+1})$. We say that $V\in\salpha_I$ if it obeys the following conditions:
    \begin{enumerate}
        \item  [\textnormal{(1)}] $V$ is stationary in $\eball_2^{n+1}$, i.e. $\delta V\equiv 0$;
        \item  [\textnormal{(2)}] $\Index(\reg(V);\eball^{n+1}_2)\leq I$;
        \item  [\textnormal{(3)}] $\H^{n-1}(\sing(V)) = 0$.
    \end{enumerate}
    Let us write $\salpha := \cup_{I\geq 0}\salpha_I$.
\end{definition}

\textbf{Remark:} From Lemma \ref{locally stable}, we know that any $V\in\salpha_I$ is locally stable and hence, from the regularity theory of Wickramasekera (\cite{wickstable}), obeys $\dim_\H(\sing(V))\leq n-7$\footnote{To be precise, when $n\leq 7$, by this we mean that for $2\leq n\leq 6$ we have $\sing(V) = \emptyset$ and for $n=7$ that $\sing(V)$ is discrete; we will write $\dim_\H(\sing(V))\leq n-7$ as a shorthand for this throughout. However, for our main results we will always have $n\geq 8$.}. In fact, it obeys $\sing(V) = \mathcal{S}_{n-7}$, i.e. it is equal to the $(n-7)^{\text{th}}$ strata (see Definition \ref{defn:strata}), and so is actually countably $(n-7)$-rectifiable by \cite{naber-valtorta}. In particular, on any ball $\Omega\subset\eball^{n+1}_2$, we know that $\reg(V)$ is two-sided in $\Omega$, and thus the second variation is only non-zero for vector fields in the normal direction of $\reg(V)$, namely, if $X = \zeta \nu$, for $\nu$ a choice of unit normal of $\reg(V)$ in $\Omega$, and $\zeta\in C^1_c(\Omega\backslash\sing(V);\R)$, then we have
$$\delta^2V(X,X) = \int_U \left\{|\nabla^V\zeta|^2 - \zeta^2|A|^2\right\}\ d\|V\|.$$
Thus, the bounded index condition in (2) is requiring that only at most $I$ linearly independent $\zeta\in C^1_c(\eball^{n+1}_2\backslash\sing(V);\R)$ can obey
\begin{equation}\label{E:index}
\int_U |A|^2\zeta^2\ d\H^n > \int_U |\nabla^V\zeta|^2\ d\H^n
\end{equation}
(here, we have removed any multiplicity from $V$, as it is constant by connectedness of $\reg(V)\cap \Omega$ and the constancy theorem). In particular, our assumption on the index in condition (2) is equivalent to requiring the ``usual'' condition that $\Index(\reg(V);\Omega)\leq I$ for each ball $\Omega\subset\eball_2^{n+1}$ with $\dim_\H(\Omega\cap \sing(V))\leq n-7$. Furthermore, as $\sing(V)$ is countably $(n-7)$-rectifiable, a simple cut-off argument (based on the fact that the 2-capacity of $\sing(V)$ vanishes) gives that, when $\Index(\reg(V))<\infty$, that $\Index(V) = \Index(\reg(V))$, where by $\Index(V)$ we mean the dimension of the maximal subspace of $\zeta\in C^1_c(B^{n+1}_2;\R)$ which obey \eqref{E:index}.

\textbf{Note:} Condition (3) in Definition \ref{s-alpha condition} can be replaced by the equivalent condition that $V$ contains no classical singularities, in the sense of \cite{wickstable}; the fact that, given (1) and (2), these two conditions are equivalent follows from the local stability provided by Lemma \ref{locally stable} and the regularity theory of \cite{wickstable}.

Our first lemma is a characterisation of the singular set of $V\in\salpha$ as the zero set of the regularity scale:

\begin{lemma}\label{regular implies positive regscale}
    Let $n\geq 2$ and $V\in \salpha$. Then, $x\in \reg(V)$ if and only if $\regscale^Q_V(x)>0$ for some positive integer $Q$.
\end{lemma}

\begin{proof}
One direction is clear: indeed, if $x\in \reg(V)$, then there is a $\rho>0$ such that $\spt\|V\|\cap \eball_\rho(x)$ is an embedded hypersurface, and furthermore is expressible as a smooth graph over the unique tangent plane $T_x\spt\|V\|$. By the constancy theorem, $V\restrictv \eball_\rho(x)$ is then a constant integer multiple of this graph, which shows that $\regscale_V^Q(x)>0$, where $Q=\Theta_V(x)\in\mathbb{Z}_{\geq 1}$.

Now suppose $x\in\spt\|V\|$ satisfies $\regscale^Q_V(x)>0$ for some $Q\in\mathbb{Z}_{\geq 1}$; thus, locally about $x$ we have that $V$ is a sum of $Q_*$ embedded (indeed graphical) $C^2$ hypersurfaces, for some $Q_*\leq Q$. In particular, as each such hypersurface has unique tangent cones at every point which are (multiplicity one) hyperplanes, this means that $V$ has a unique tangent cone $\mathbf{C}$ at $x$ which is supported on a union of hyperplanes. However, the minimal distance theorem of Wickramasekera (\cite{wickstable}*{Theorem 3.4}, which we can apply as $V$ is stable locally about $x$ by assumption of it being a union of codimension one graphs) implies that in fact $\mathbf{C}$ must be supported on a \textit{single} hyperplane (as if there were more than one, this would create a classical singularity in $\mathbf{C}$, contradiction the minimal distance theorem for $V$). But now we can apply Wickramasekera's sheeting theorem (\cite{wickstable}*{Theorem 3.3}) to see that in fact $x\in\reg(V)$, which completes the proof.
\end{proof}

\textbf{Remark:} All we needed for the above proof was stationarity of $V$ and $\H^{n-1}(\sing(V)) = 0$; we did not need the bounded index assumption.

Let us recall the main compactness theorem of Wickramasekera's regularity theory, which applies to the class $\salpha_0$:

\begin{theorem}[\cite{wickstable}*{Theorem 3.1}]\label{compactness theorem}
    Let $n\in \mathbb{Z}_{\geq 2}$. Suppose $(V_i)_{i=1}^\infty\subset \salpha_0$ is a sequence satisfying $\sup_i\|V_i\|(\eball_2^{n+1})<\infty$. Then, there exists a subsequence $(V_{i_j})_{j=1}^\infty$ and $V\in \salpha_0$ such that $V_{i_j}\to V$ as varifolds in $\eball^{n+1}_2$ and smoothly (i.e. in the $C^k$ topology for each $k$) locally in $\eball^{n+1}_2\backslash\sing(V)$.
\end{theorem}

We also recall the so-called sheeting theorem of Wickramasekera, stated in terms of the regularity scale previously defined:

\begin{theorem}[Sheeting Theorem, \cite{wickstable}*{Theorem 3.3}]\label{sheeting theorem}
    Let $n\in \mathbb{Z}_{\geq 2}$, $\Lambda>0$, and $\theta\in (0,1)$. Then, there exists $\varepsilon_0 = \varepsilon_0(n,\Lambda,\theta)\in (0,1)$ and $Q_0 = Q_0(n,\Lambda,\theta)\in \mathbb{Z}_{\geq 1}$ such that the following is true: whenever $V\in \salpha_0$ with $0\in \spt\|V\|$ satisfies:
    \begin{enumerate}
        \item [\textnormal{(a)}] $(2^n\omega_n)^{-1}\|V\|(\eball^{n+1}_2)\leq \Lambda$;
        \item [\textnormal{(b)}] $\textnormal{dist}_\H(\spt\|V\|\cap (B^n_1(0)\times\R), B^n_1(0))<\varepsilon_0$;
    \end{enumerate}
    then we have $\regscale^{Q_0}_V(0)\geq \theta$; here, $\omega_n := \H^n(B_1^n(0))$ and $\textnormal{dist}_\H$ denotes the Hausdorff distance.
\end{theorem}

\begin{proof}
    It follows from Schauder theory and the sheeting theorem of Wickramasekera (\cite{wickstable}*{Theorem 3.3}) that there exists $\varepsilon_0 = \varepsilon_0(n,\Lambda,\theta)\in (0,1)$ and $C_1 = C_1(n,\Lambda,\theta)\in (0,\infty)$ such that if the above assumptions hold, then there is a $Q\in\mathbb{Z}_{\geq 1}$ and smooth functions $u_i:B_\theta^n(0)\to \R$ for $i=1,\dotsc,Q$ such that
    $$V\restrictv (B^n_\theta(0)\times\R) = \sum^Q_{i=1}\mathbf{v}(u_i)$$
    and moreover that $\|u_i\|_{C^4}\leq C_1\varepsilon_0$ for each $i=1,\dotsc,Q$. Clearly we have $\|\mathbf{v}(u_i)\|(B_\theta^{n}(0)\times\R) \geq \H^n(B_\theta^{n}(0))$ for each $i=1,\dotsc,Q$, which gives $Q\omega_n\theta^n \leq (2^n\omega_n)\Lambda$, i.e. $Q\leq Q_0$, where $Q_0 = Q_0(n,\Lambda,\theta)$. We can clearly choose $\varepsilon_0$ small enough to guarantee that $\graph(u_i)\cap \eball^{n+1}_\theta(0)\neq\emptyset$ for all $i=1,\dotsc,Q$, and thus we have $\regscale^Q_{0,V}(0)\geq \theta$ (as defined in Definition \ref{regularity scale}). From our estimate on $\|u_i\|_{C^4}$ for each $i=1,\dotsc,Q$, we clearly have for any $\rho\in (0,\theta)$,
    $$\sup_{\spt\|V\|\cap B_\rho^{n+1}(0)}\rho|A_V| \leq \theta Q_0 C_1\varepsilon_0$$
    and so this is $\leq 1$ when $\theta Q_0C_1\varepsilon_0 \leq 1$; this can be guaranteed by taking $\varepsilon_0 = \varepsilon_0(n,\Lambda,\theta)$ smaller if necessary. Thus, this shows that $\regscale^Q_V(0)\geq \theta$, completing the proof, as $\regscale^{Q_0}_V(0)\geq \regscale^Q_V(0)$.
\end{proof}


\section{Quantitative Stratification}\label{sec:strata}

In this section we recall the notion of \textit{strata} as well as the \textit{quantitative strata} in codimension one, and prove a relation (essentially an $\varepsilon$-regularity theorem) between the regularity scale and certain strata.

\begin{definition}
    Let $\BC\in\ivarifolds_n(\R^{n+1})$. We say that $\BC$ is a cone if $(\eta_{0,r})_\#\BC = \BC$ for all $r>0$.
\end{definition}

Given a cone $\BC\in \ivarifolds_n(\R^{n+1})$, we define the \textit{spine} of $\BC$, denoted $S(\BC)$, to be the set of points along which $\BC$ is translation invariant, i.e.
$$S(\BC):= \{x\in \R^{n+1}:(\eta_{x,1})_\#\BC = \BC\}.$$
It is simple to check that $S(\BC)\subset\R^{n+1}$ is a subspace; thus, $\dim(S(\BC))\leq n$. We then say that $\BC$ is $k$\textit{-symmetric} if $\dim(S(\BC)) \geq k$. Let us write $\kcone_k\subset\ivarifolds_n(\R^{n+1})$ for the set of (codimension one) $k$-symmetric cones. Note that when $S(\BC)\neq \spt\|\BC\|$ (i.e. $\BC$ is not supported on a hyperplane), then $S(\BC)\subset \sing(\BC)$.

Next we define a quantitative notion of being \textit{almost conical} for varifolds, in the same way as seen in \cite{naber-valtorta}*{Definition 1.1} and \cite{cheeger-naber2013}*{Definition 5.3}.

\begin{definition}\label{defn:conical}
    Fix $U\subset\R^{n+1}$ open, and let $\delta>0$, $r>0$, and $k\in \{0,1,\dotsc,n\}$. We say that $V\in\ivarifolds_n(U)$ is $(\delta,r,k)$-conical at a point $x\in \spt\|V\|$ if $\eball_r(x)\subset U$ and there exists a $k$-symmetric cone $\BC\in\kcone_k$ such that
    $$\mathbf{d}\left((\eta_{x,r})_\#V\restrictv \eball_1,\BC\restrictv\eball_1\right)\leq \delta;$$
    here, we recall $\mathbf{d}$ is the metric corresponding to the Fr\' echet structure of varifold topology, which induces the same topology.
\end{definition}

We now define various stratifications of the singular set, as in \cite{naber-valtorta}*{Definition 1.2}:

\begin{definition}\label{defn:strata}
    Fix $U\subset\R^{n+1}$ open and let $\delta>0$, $R\in (0,1]$, $r\in (0,R)$, and $V\in \ivarifolds_n(U)$ with bounded first variation. Then for each $k\in\{0,\dotsc,n\}$, we define:
    \begin{enumerate}
        \item [\textnormal{(1)}] The $k^{\text{th}}$ $(\delta,r,R)$-stratification by:
        $$\strata^k_{\delta,r,R}(V):= \{x\in\spt\|V\|: V\text{ is not }(\delta,s,k+1)\text{-conical at }x\text{ for all }s \in [r,R)\};$$
        \item [\textnormal{(2)}] The $k^{\text{th}}$ $\delta$-stratification by:
        $$\strata^k_\delta(V):= \bigcap_{0<r<1}\strata^k_{\delta,r,1}(V);$$
        \item [\textnormal{(3)}] The $k^{\text{th}}$ stratification by:
        $$\strata^k(V):= \bigcup_{\delta>0}\strata^k_\delta(V).$$
    \end{enumerate}
\end{definition}

\textbf{Note:} We have $\strata^k(V) \equiv \{x\in\spt\|V\|: \dim(S(\BC))\leq k\text{ for each }\BC\in\tangentcones_x(V)\}$.

By definition, $\strata^0(V)\subset\cdots\subset\strata^{n-1}(V)\subset\strata^n(V)\equiv \spt\|V\|$, and $\strata^{n-1}(V)\subset\sing(V)$. We remark that it is well-known (see \cite{almgren} or \cite{simon-cylindrical}*{(1.10)}) that $\dim_\H(\strata^k(V))\leq k$ for each $k\in\{0,1,\dotsc,n\}$, and moreover that $\strata^k(V)$ is countably $k$-rectifiable (see \cite{naber-valtorta}). We have already remarked that in the above (codimension one) situation, for $V\in \salpha$ we have $\sing(V) = \strata^{n-7}(V)$ (which follows from the regularity theory \cite{wickstable} along with the local stability provided by Lemma \ref{locally stable}); in particular, this applies to codimension one area minimisers.

We note the following immediate consequence of the definition:

\begin{lemma}\label{rescaled strata}
    Let $U\subset\R^{n+1}$ be open, $\delta>0$, $0<r<R\leq 1$, $k\in\{0,1,\dotsc,n\}$ and $V\in\ivarifolds_n(U)$ with bounded first variation. Then for any $x\in \spt\|V\|$ and $\rho>0$ with $\eball_\rho(x)\subset U$, we have
    $$\eta_{x,\rho}\left(\strata^k_{\delta,r,R}(V)\right) = \strata^k_{\delta,r/\rho,R/\rho}\left((\eta_{x,\rho})_\#V\right).$$
\end{lemma}

\begin{proof}
This is immediate by definition, as $\eta_{y,s}\circ\eta_{x,\rho} = \eta_{x+\rho y,s\rho}$, and so $y\in \strata^k_{\delta,r/\rho,R/\rho}\left((\eta_{x,\rho})_\#V\right)$ if and only if $x+\rho y\in \strata^k_{\delta,r,R}(V)$.
\end{proof}

The following $\varepsilon$-regularity result is the corresponding version of \cite{cheeger-naber2013}*{Theorem 6.2} for locally stable varifolds instead of codimension one area minimising currents.

\begin{theorem}[$\varepsilon$-Regularity Theorem]\label{epsilon regularity}
    Let $n\in\mathbb{Z}_{\geq 2}$, $\Lambda\in (0,\infty)$, and $K\subset\eball^{n+1}_2$ compact. Then, there exist constants $\varepsilon_0 = \varepsilon_0(n,\Lambda,K)\in (0,1)$ and $Q_0 = Q_0(n,\Lambda,K)\in \mathbb{Z}_{\geq 1}$ such that the following holds: if $V\in \ivarifolds_n(\eball^{n+1}_2)$ is a stationary integral varifold, $x\in \spt\|V\|\cap K$, and $\rho\in (0,d(x,\del\eball^{n+1}_2)]$ satisfy:
    \begin{enumerate}
        \item [\textnormal{(a)}] $\|V\|(\eball^{n+1}_2)\leq \Lambda$;
        \item [\textnormal{(b)}] $(\eta_{x,\rho/2})_\#V\in \salpha_0$;
        \item [\textnormal{(c)}] $V$ is $(\varepsilon_0,\rho/2,n-6)$-conical at $x$;
    \end{enumerate}
    then we have $\regscale^{Q_0}_V(x)\geq \rho/4$.
\end{theorem}

\begin{proof}
We argue this by contradiction, so suppose the claim was false. Then, we can find $\Lambda>0$ and a compact set $K\subset\eball^{n+1}_2$ such that for each $k\in \mathbb{Z}_{\geq 1}$, we can find a varifold $V_k\in \ivarifolds_n(\eball^{n+1}_2)$, a point $x_k\in \spt\|V_k\|\cap K$, and $\rho_k \in (0,d(x_k,\del\eball^{n+1}_2)]$ such that $\|V_k\|(\eball^{n+1}_2)\leq \Lambda$, $(\eta_{0,\rho_k/2})_\#V_k \in \salpha_0$, with $V_k$ being $(1/k,\rho_k/2,n-6)$-conical at $x_k$, yet $\regscale^k_{V_k}(x_k)<\rho_k/4$.

Now set $W_k:= (\eta_{x_k,\rho_k/2})_\#V_k$. Then we have $W_k\in \salpha_0$, and
\begin{align*}
\|W_k\|(\eball_2^{n+1}) \equiv (\rho_k/2)^{-n}\|V_k\|(B_{\rho_k}(x_k))  & \leq 2^n\cdot d(x_k,\del\eball^{n+1}_2)^{-n}\|V_k\|\left(\eball_{d(x_k,\del\eball^{n+1}_2)}(x_k)\right)\\
& \leq 2^n\cdot d(K,\del\eball^{n+1}_2)^{-n}\Lambda
\end{align*}
where in the second inequality here we have used the monotonicity formula for stationary integral varifolds. Moreover, $W_k$ is $(1/k,1,n-6)$-conical at $0$, and $\regscale^k_{W_k}(0) = (\rho_k/2)^{-1}\regscale^k_{V_k}(x_k)<1/2$. 

In particular, $(W_k)_k$ satisfies the conditions of Theorem \ref{compactness theorem}, and so there is a subsequence $(W_{k_j})_j$ and $W\in \salpha_0$ such that $W_{k_j}\weakly W$ as varifolds in $\eball^{n+1}_2(0)$. However, as $W_k$ is $(1/k,1,n-6$)-conical for each $k$, this implies that $W$ must be a cone, with $\dim_\H(S(W))\geq n-6$. However, as $W\in \salpha_0$, we know $\dim_\H(\sing(W))\leq n-7$, and thus this implies that $W$ must be a hyperplane of some integer multiplicity. But then, as the $W_k$ are stationary and so varifold convergence implies local convergence of the supports in Hausdorff distance (see, e.g. \cite{mondino}*{Proposition 3.9}), it follows from Theorem \ref{sheeting theorem} (with $\theta=1/2$) that for all $j$ sufficiently large, we have $\regscale^{Q_0}_{W_{k_j}}(0) \geq 1/2$ for some $Q_0 = Q_0(n,\Lambda,K)$; this is a direct contradiction to the fact that $\regscale^{k_j}_{W_{k_j}}(0)<1/2$ for all $j$ sufficiently large. Hence, the proof is completed.
\end{proof}

\textbf{Remark:} The above proof is slightly different from that for codimension one area minimisers seen in \cite{cheeger-naber2013}, as it does not require estimates on the various stratifications. We also note that the above proof shows that the dependence on the compact set $K\subset\eball^{n+1}_2$ of the constants $\varepsilon_0$ and $Q_0$ in Theorem \ref{epsilon regularity} is in fact only on a lower bound for $d(K,\del\eball^{n+1}_2(0))$ (the same is true for the next Corollary also).

Finally, we remark the following corollary of Theorem \ref{epsilon regularity}, which is essentially a rephrasing of the result when $n\geq 7$ in terms of the quantitative strata.

\begin{corollary}\label{corollary of sheeting theorem}
    Let $n\in \mathbb{Z}_{\geq 7}$, $\Lambda\in (0,\infty)$, and $K\subset\eball^{n+1}_2$ be a compact subset. Then, there exist constants $\varepsilon_0 = \varepsilon_0(n,\Lambda,K)\in (0,1)$ and $Q_0 = Q_0(n,\Lambda,K)\in \mathbb{Z}_{\geq 1}$ such that the following is true: if $V\in \ivarifolds_n(\eball^{n+1}_2)$ is a stationary integral varifold obeying:
    \begin{enumerate}
        \item [\textnormal{(a)}] $\|V\|(\eball^{n+1}_2)\leq \Lambda$;
        \item [\textnormal{(b)}] $(\eta_{x,d(K,\del\eball_2^{n+1})})_\#V\in \salpha_0$ for all $x\in K$;
    \end{enumerate}
    then we have $\badreg^{Q_0}_{\sigma/5}(V)\cap K\subset \strata^{n-7}_{\varepsilon_0,\sigma/2,d(K,\del\eball^{n+1}_2)}(V)\cap K$ for all $\sigma\in (0,d(K,\del\eball^{n+1}_2)]$.
\end{corollary}

\begin{proof}
    Let $\varepsilon_0 = \varepsilon_0(n,\Lambda,K)\in(0,1)$ and $Q_0 = Q_0(n,\Lambda,K)\in \mathbb{Z}_{\geq 1}$ be as in Theorem \ref{epsilon regularity}. If the result were not true with this choice of $\varepsilon_0$ and $Q_0$, then we could find $\sigma\in (0,d(K,\del\eball^{n+1}_2)]$ and $x\in \badreg^{Q_0}_{\sigma/5}(V)\cap K$ with $x\not\in \strata^{n-7}_{\varepsilon,\sigma/2,d(K,\del\eball^{n+1}_2)}(V)$, i.e. there exists some $s\in [\sigma,d(K,\del\eball^{n+1}_2))$ for which $V$ is $(\varepsilon_0,s/2,n-6)$-conical at $x$. Hence, as $\eball_s(x)\subset \eball_{d(K,\del\eball^{n+1}_2)}(x)$, we may apply Theorem \ref{epsilon regularity} to see that $\regscale^{Q_0}_V(x)\geq {\sigma/4}$, which is a contradiction.
\end{proof}

\section{Singular set estimates}\label{sec:estimates}

In this section we prove Theorem \ref{thm:A}, i.e. local estimates for measure of the tubular neighbourhood of the singular set of codimension one stationary integral varifolds in Euclidean space which have regular part of bounded index.

The following result is the main measure bound on the quantitative strata from \cite{naber-valtorta} (in the codimension one setting):

\begin{theorem}[\cite{naber-valtorta}*{Theorem 1.3 and 1.4}]\label{quantitative estimates bounded variation}
    Let $n\in\mathbb{Z}_{\geq 2}$, $\Lambda\in (0,\infty)$, $H\in (0,\infty)$, and $\varepsilon\in (0,1)$. Then there exists a constant $C_\varepsilon = C_\varepsilon(n,\Lambda,H,\varepsilon)\in (0,\infty)$ such that the following is true: if $V\in\ivarifolds_n(\eball^{n+1}_2)$ has $\|V\|(\eball^{n+1}_2)\leq\Lambda$ and has first variation bounded by $H$ in $\eball^{n+1}_2(0)$, then for each $k\in \{0,1,\dotsc,n\}$ we have
    $$\H^{n+1}\left(\eball_r(\strata^k_{\varepsilon,r,1}(V))\cap \eball_1\right) \leq C_\varepsilon r^{n+1-k}\ \ \ \ \text{for all }r\in (0,1].$$
\end{theorem}

We remark that Theorem \ref{quantitative estimates bounded variation} has an immediate rescaled version, which we state in the special case of \textit{stationary} integral varifolds for sake of simplicity (as it is all we will need for our results):

\begin{corollary}\label{rescaled quantitative estimates}
    Let $n\in\mathbb{Z}_{\geq 2}$, $\Lambda\in (0,\infty)$, and $\varepsilon\in (0,1)$. Then, there exists a constant $C_\varepsilon = C_\varepsilon(n,\Lambda,\varepsilon)\in (0,\infty)$ such that the following is true: for any $R\in (0,1/2]$ and any stationary integral varifold $V\in \ivarifolds_n(\eball^{n+1}_2)$ with $\|V\|(\eball_2^{n+1}) \leq \Lambda$, we have
    $$\H^{n+1}\left(\eball_r(\strata^k_{\varepsilon,r,R}(V))\cap \eball_R(x)\right)\leq C_\varepsilon r^{n+1-k}R^k\ \ \ \ \text{for all }r\in (0,R]\text{ and }x\in\eball_1^{n+1}.$$
\end{corollary}

\begin{proof}
    Since $\eball^{n+1}_{2R}(x)\subset\eball_2^{n+1}$, $W:= (\eta_{x,R})_\#V$ is a stationary integral varifold in $\eball^{n+1}_2$, and moreover from the monotonicity formula we know that 
    $$\|W\|(\eball_2^{n+1}) \equiv R^{-n}\|V\|(\eball^{n+1}_{2R}(p))\leq 2^n\|V\|(B_1(x))\leq 2^n\Lambda.$$
    By Lemma \ref{rescaled strata} with $\rho = R$, we know that $\eta_{x,R}\left(\strata_{\varepsilon,r,R}(V)\right) = \strata^k_{\varepsilon,r/R,1}(W)$. Thus, if $C_\varepsilon = C_\varepsilon(n,2^n\Lambda,0,\varepsilon)\in (0,\infty)$ is the constant from Theorem \ref{quantitative estimates bounded variation} with these parameters, then we know from Theorem \ref{quantitative estimates bounded variation} that
    $$\H^{n+1}\left(\eball_r(\strata^k_{\varepsilon,\rho,1}(W))\cap\eball_1\right) \leq C_\varepsilon \rho^{n+1-k}\ \ \ \ \text{for all }\rho\in (0,1].$$
    Hence,
    \begin{align*}
        \H^{n+1}\left(\eball_r(\strata^k_{\varepsilon,r,R}(V))\cap\eball_R(x)\right) & = \H^{n+1}\left[\eball_r\left(\eta^{-1}_{x,R}(\strata^k_{\varepsilon,r/R,1}(W))\right)\cap \eball_R(x)\right]\\
        & = \H^{n+1}\left[\eta^{-1}_{x,R}\left(\eball_{r/R}(\strata^k_{\varepsilon,r/R,1}(W))\right)\cap \eball_R(x)\right]\\
        & = \H^{n+1}\left[\eta^{-1}_{x,R}\left(\eball_{r/R}(\strata^k_{\varepsilon,r/R,1}(W))\cap \eball_1\right)\right]\\
        & = R^{n+1}\H^{n+1}\left(\eball_{r/R}(\strata^k_{\varepsilon,r/R,1}(W))\cap \eball_1\right)\\
        & \leq R^{n+1}\cdot C_\varepsilon(r/R)^{n+1-k}\\
        & = C_\varepsilon r^{n+1-k}R^k
    \end{align*}
    as desired.
\end{proof}

Our first estimate toward proving Theorem \ref{thm:A} is independent of the Naber--Valtorta measure estimates of Theorem \ref{rescaled quantitative estimates} and doesn't require any assumption on the size of the singular set. It applies for tubular neighbourhoods about the points where the stability radius is at most the size of the tubular radius; morally, this is when the stability radius is ``small''. We will see that in fact we can get a much better bound on this set:

\begin{lemma}\label{low stability bound}
    Let $n\in\mathbb{Z}_{\geq 2}$. Then, there exists $C_0 = C_0(n)\in (0,\infty)$ such that for any stationary integral varifold $V\in \ivarifolds_n(\eball^{n+1}_2)$ with $\Index(\reg(V))<\infty$, we have
    $$\H^{n+1}\left(\eball_r(s_V^{-1}(0,r))\cap \eball_{1/2}\right) \leq C_0\folding(V)r^{n+1}\ \ \ \ \text{for all }r\in (0,1/2].$$
\end{lemma}

\begin{proof}
Set $A:= s_V^{-1}(0,r)\cap \eball^{n+1}_1$. Then as $\mathcal{A}:= \{\overline{\eball}_{s_V(a)}(a): a\in A\}$ is a Besicovitch cover of $A$, the Besicovitch covering theorem gives the existence of a constant $N_0 = N_0(n)$ and subcollections $A_1,\dotsc,A_N\subset A$, where $N\leq N_0$, for which each individual collection $\mathcal{A}_j:= \{\overline{\eball}_{s_V(a)}:a\in A_j\}$ consists of pairwise disjoint balls, and moreover we still have
$$A\subset\bigcup^N_{j=1}\bigcup_{a\in A_j}\overline{\eball}_{s_V(a)}(a).$$
We now claim that $|A_j|\leq \folding(V)$ for each $j=1,\dotsc,N$ (recall that $\folding(V)\leq \Index(\reg(V);\eball_2^{n+1})$, so in particular $\folding(V)<\infty$ by assumption). Indeed, if this failed, we would be able to find distinct $a_1,\dotsc,a_{\folding(V)+1}$ in some $A_j$. But then, as $\overline{B}_{s_V(a_i)}(a_i)$ are pairwise disjoint for $i=1,\dotsc,\folding(V)+1$, we can find $\delta>0$ for which $\eball_{s_V(a_i)+\delta}(a_i)$ are pairwise disjoint for $i=1,\dotsc,\folding(V)+1$. But by definition of $s_V$, $V$ is unstable in each $\eball_{s_V(a_i)+\delta}(a_i)$, and so this would imply that $\folding(V)\geq \folding(V)+1$, which is clearly a contradiction as $\folding(V)<\infty$ (by Lemma \ref{bounded index disjoint sets}). Thus, $|A_j|\leq \folding(V)$ for each $j=1,\dotsc,N$.

Now, as $s_V(a)<r$ for all $a\in A$, we then have for any $\varepsilon>0$,
$$\eball_r(A)\subset\bigcup^N_{j=1}\bigcup_{a\in A_j}\eball_{2r+\varepsilon}(a)$$
and therefore:
$$\H^{n+1}(\eball_r(A))\leq N\cdot\folding(V)\cdot\omega_{n+1}(2r+\varepsilon)^{n+1}.$$
Hence the result follows, with $C_0 := 2^{n+1}N\omega_{n+1}$, by taking $\varepsilon\downarrow 0$ and then noting that $\eball_r(s_V^{-1}(0,r))\cap \eball_{1/2}\subset \eball_r(A)$ for $r\in (0,1/2]$.
\end{proof}

The next lemma is the key covering argument we need to prove Theorem \ref{thm:A}. The idea is that control on the folding number provides control on the size of a covering by smaller stable balls.

\begin{lemma}\label{covering lemma}
    Let $n\in \mathbb{Z}_{\geq 2}$. Then, there exists a constant $C_1 = C_1(n)\in (0,\infty)$ such that the following holds: if $V\in\ivarifolds_n(\eball^{n+1}_2)$ is a stationary integral varifold with $\folding(V)<\infty$, then for any $0<a<b<\infty$, $\gamma\in (0,1)$, and subset $A\subset s_V^{-1}([a,b])\cap \eball_1$, there exists a finite set $B\subset A$ such that
    $$A\subset\bigcup_{y\in B}\eball^{n+1}_{\gamma s_V(y)}(y)$$
    and moreover we have the size bound
    $$|B|\leq C_1\folding(V)\cdot\left(\frac{b}{a}\right)^{n+1}\left(1+\gamma^{-1}\right)^{n+1}.$$
\end{lemma}

\begin{proof}
We first prove that for such $A$ we have $\H^{n+1}(B_r(A))\leq C\folding(V)\omega_{n+1}(b+r)^{n+1}$ for some $C = C(n)$; this will follow in much the same way it did in Lemma \ref{low stability bound}. Indeed, consider the Besicovitch covering of $A$ given by $\mathcal{A}:= \{\overline{\eball}_{s_V(x)}(x):x\in A\}$. The Besicovitch covering theorem then gives the existence of a constant $N_0 = N_0(n)$ and (countable) subcollections $A_1,\dotsc,A_N\subset A$, where $N\leq N_0$ and, for each $i=1,\dotsc,N$, $\{\overline{\eball}_{s_V(x)}(x):x\in A_i\}$ is formed of pairwise disjoint balls, and moreover
$$A\subset\bigcup^N_{i=1}\bigcup_{x\in A_i}\overline{\eball}_{s_V(x)}(x).$$
We then claim that in fact $|A_i|\leq \folding(V)$ for each $i=1,\dotsc,N$. Indeed, if not, then we can find $i$ (which without loss of generality we can assume to be $i=1$) such that $|A_1|\geq \folding(V)+1$. In particular, we can find points $x_1,\dotsc,x_{\folding(V)+1}\in A_1$, and for these points we can then find a $\delta>0$ for which $\eball_{s_V(x_1)+\delta}(x_1),\dotsc,\eball_{s_V(x_{\folding(V)+1})+\delta}(x_{\folding(V)+1})$ are disjoint. As $\reg(V)$ is unstable in each $\eball_{s_V(x_j)+\delta}(x_j)$ for $j=1,\dotsc,\folding(V)+1$, this contradicts the definition of $\folding(V)$ (as again, by assumption and Lemma \ref{bounded index disjoint sets} we know $\folding(V)<\infty$); hence this shows that $|A_i|\leq \folding(V)$ for each $i=1,\dotsc,N$. Hence, as we have that for any $r>0$ 
$$\eball_{r}(A)\subset\bigcup_{i=1}^N\bigcup_{x\in A_i}\overline{\eball}_{s_V(x)+r}(x),$$
this immediately gives
\begin{equation}\label{E:cover-1}
\H^{n+1}(\eball_r(A)) \leq N\cdot \folding(V)\cdot \omega_{n+1}(b+r)^{n+1}
\end{equation}
as here we have $s_V(x)\leq b$ for all $x\in A$.

We now use \eqref{E:cover-1} to prove the lemma. Consider a new Besicovitch cover of $A$ given by $\mathcal{B}:= \{\overline{\eball}_{\frac{1}{2}\gamma s_V(x)}(x): x\in A\}$. Applying the Besicovitch covering theorem again, we again get the existence of $M\leq N_0$ and (countable) subcollections $B_1,\dotsc,B_M \subset A$ where, for $i=1,\dotsc,M$, $\{\overline{B}_{\frac{1}{2}\gamma s_V(x)}(x): x\in B_i\}$ consists of pairwise disjoint balls, and moreover
\begin{equation}\label{E:cover-3}
A\subset\bigcup^M_{i=1}\bigcup_{x\in B_i}\overline{\eball}_{\frac{1}{2}\gamma s_V(x)}(x).
\end{equation}
We now claim that each $B_i$ is a finite set, and moreover for each $i=1,\dotsc,M$ we have
\begin{equation}\label{E:cover-2}
|B_i|\leq C\folding(V)\left(\frac{b}{a}\right)^{n+1}(1+\gamma^{-1})^{n+1}.
\end{equation}
To see this, consider any finite subset $\{x_1,\dotsc,x_L\}\subset B_i$ (for some $i=1,\dotsc,M$). Then, clearly have
$$\bigcup^L_{j=1}\eball_{\frac{1}{2}\gamma s_V(x_j)}(x_j)\subset\bigcup^L_{j=1}\eball_{\frac{1}{2}\gamma b}(x_j)\subset \eball_{\gamma b}(A).$$
Thus, as the left-hand side consists of pairwise disjoint balls, and thus from \eqref{E:cover-1} with $r= b\gamma$ gives
$$L\omega_{n+1} \left(\frac{1}{2}\gamma a\right)^{n+1} \leq C\folding(V)\omega_{n+1}(b + b\gamma)^{n+1}$$
i.e. $L \leq C\folding(V)\left(\frac{b}{a}\right)^{n+1}(1+\gamma^{-1})^{n+1}$, where $C = C(n)$; this proves \eqref{E:cover-2}. Thus, as $\overline{\eball}_{\frac{1}{2}\gamma s_V(x)}(x)\subset \eball_{\gamma s_V(x)}(x)$, the claim follows from \eqref{E:cover-3} by taking $B:= \cup_{i=1}^M B_i$.
\end{proof}

We can now prove Theorem \ref{thm:A}. Since we know $V\in \salpha$ have that $\sing(V)$ is countably $(n-7)$-rectifiable (see the remark following Definition \ref{s-alpha condition}), all that remains is to prove the following:

\begin{theorem}\label{thm:A-2}
    Let $n\in \mathbb{Z}_{\geq 8}$, $\Lambda>0$. Then, there exists $C_0 = C_0(n,\Lambda)\in (0,\infty)$ such that for any $V\in \salpha$ with $\|V\|(\eball^{n+1}_2)\leq \Lambda$, we have for all $r\in (0,1/2]$
    $$\H^{n+1}\left(\eball_{r/8}(\sing(V))\cap \eball_{1/2}\right)\leq C_0(1+\folding(V))r^8$$
    $$\|V\|\left(\eball_{r/8}(\sing(V))\cap \eball_{1/2}\right)\leq C_0(1+\folding(V))r^7.$$
\end{theorem}

\begin{proof}
For $r\in (0,1/2]$. Observe that $\eball_{r/8}(\sing(V))\cap \eball_{1/2}\subset \eball_{r/8}(\sing(V)\cap \eball_1)$. Write
$$\sing(V) = \left(s_V^{-1}(0,r)\cup s_V^{-1}([r,1])\cup s_V^{-1}((1,\infty])\right)\cap \sing(V)$$
and so
\begin{align*}
\H^{n+1}\left(\eball_{r/8}(\sing(V))\cap \eball_{1/2}\right) & \leq \H^{n+1}\left(\eball_{r/8}(s_V^{-1}(0,r)\cap \sing(V))\cap \eball_{1/2}\right)\\
& + \H^{n+1}\left(\eball_{r/8}(s_V^{-1}([r,1])\cap \sing(V)\cap \eball_{1})\right)\\
& + \H^{n+1}\left(\eball_{r/8}(s_V^{-1}((1,\infty])\cap \sing(V))\cap \eball_{1/2}\right)
\end{align*}
We will bound each term individually. Note firstly that Lemma \ref{low stability bound} gives
$$\H^{n+1}\left(\eball_{r/8}(s_V^{-1}(0,r)\cap\sing(V))\cap \eball_{1/2}\right) \leq C_0\folding(V)r^{n+1}$$
where $C_0 = C_0(n)\in (0,\infty)$; this deals with the first term. 

For the third term, note that if $x\in s_V^{-1}((1,\infty])\cap \sing(V)\cap \eball_1$, then in particular $x\in \sing(V)$, and so Lemma \ref{regular implies positive regscale} gives $\regscale^Q_V(x)=0$ for some positive integer $Q$. Hence, if we take $K = \overline{s_V^{-1}((1,\infty])\cap \eball_1}$ in Corollary \ref{corollary of sheeting theorem}, it gives that
$$s_V^{-1}((1,\infty])\cap \sing(V)\cap \eball_1 \subset \strata^{n-7}_{\varepsilon_0,r,1/2}(V)$$
for each such $r$, where $\varepsilon_0 = \varepsilon_0(n,\Lambda)$. Hence,
\begin{align*}
\H^{n+1}\left(\eball_{r/8}(s_V^{-1}((1,\infty])\cap\sing(V))\cap \eball_{1/2}\right) & \leq \H^{n+1}(\eball_{r/8}(\strata^{n-7}_{\varepsilon_0,r,1/2})\cap \eball_{1/2})\\
& \leq C_{\varepsilon_0}r^{8}(1/2)^{n-7}
\end{align*}
where $C_{\varepsilon_0} = C_{\varepsilon_0}(n,\Lambda,\varepsilon_0)$ only depends on $n$ and $\Lambda$. Hence, all that remains to bound is the second term above, where the stability radius is in $[r,1]$.

Set $k_0:= \min\{k\in\mathbb{Z}_{\geq 1}:2^{-k}\leq r\}$. Then for $j=1,\dotsc,k_0$, set
$$A_j:= s_V^{-1}([2^{-j},2^{-j+1}])\cap\sing(V)\cap \eball_1.$$
Then note that
$$s_V^{-1}([r,1])\cap\sing(V)\cap\eball_1\subset A_1\cup\cdots \cup A_{k_0}.$$
Now, applying Lemma \ref{covering lemma} with $a = 2^{-j}$, $b=2^{-j+1}$, $\gamma = 1/8$, we obtain for each $j=1,\dotsc,k_0$ a set $B_j\subset A_j$ with $|B_j|\leq 16^{n+1}C_1\folding(V)$, where $C_1 = C_1(n)\in (0,\infty)$, and
$$A_j\subset \bigcup_{y\in B_j}\eball_{s_V(y)/8}(y).$$
Observe that, by continuity of the stability radius (Lemma \ref{stability radius continuity}), we know that $s_V\geq 2^{-j}$ on $\overline{A}_j$, and thus for each $x\in \overline{A}_j$ we know that $\reg(V)$ is stable in $\eball_{2^{-j}}(x)$. Hence, using Lemma \ref{regular implies positive regscale} and Corollary \ref{corollary of sheeting theorem} (with $K = \overline{A}_j$, which has $d(K,\del \eball^{n+1}_2)\geq 1$, $\sigma = 5r/8$) we have that
$$A_j\subset \strata^{n-7}_{\varepsilon_0,r/8,2^{-j-1}}(V)$$
where here $\varepsilon_0 = \varepsilon_0(n,\Lambda)$ (note that $r/8 < (1/8)2^{-k_0+1} = 2^{-k_0-2}<2^{-j-1}$ for any $j=1,\dotsc,k_0$). Thus, we in particular have
$$\eball_{r/8}(A_j)\subset\bigcup_{y\in B_j}\eball_{r/8}\left(\strata^{n-7}_{\varepsilon_0,r/8,2^{-j-1}}(V)\right)\cap \eball_{s_V(y)/8 + r/8}(y)$$
and as $s_V(y)\leq 2^{-j+1}$ in $A_j$,
$$\eball_{r/8}(A_j)\subset\bigcup_{y\in B_j}\eball_{r/8}\left(\strata^{n-7}_{\varepsilon_0,r/8,2^{-j-1}}(V)\right)\cap \eball_{2^{-j-1}}(y).$$
Hence, by Corollary \ref{rescaled quantitative estimates}, we can find a constant $C^\prime = C^\prime(n,\Lambda)$ (for this choice of $\varepsilon_0 = \varepsilon_0(n,\Lambda)$) such that for each $y\in B_j$ we have
$$\H^{n+1}\left(\eball_{r/8}\left(\strata^{n-7}_{\varepsilon_0,r/8,2^{-j-1}}(V)\right)\cap \eball_{2^{-j-1}}(y)\right) \leq C^\prime\cdot\left(r/8\right)^8\cdot (2^{-j-1})^{-(n-7)}.$$
Hence, we have by combining
$$\H^{n+1}(\eball_{r/8}(A_j)) \leq |B_j|\cdot C^\prime\cdot (r/8)^8 \cdot (2^{-j-1})^{-(n-7)} \leq C_* \folding(V)r^8\cdot \left(2^{-n+7}\right)^{-j-1}$$
and this is true for each $j=1,\dotsc,k_0$; here $C_* = C_*(n,\Lambda)$. Therefore,
$$\H^{n+1}\left(\eball_{r/8}(A_1\cup\cdots\cup A_{k_0})\right) \leq C_*\folding(V)r^8\cdot\sum^{k_0}_{j=1}\left(\frac{1}{2^{n-7}}\right)^{-j-1}.$$
Since $n\geq 8$, the sum on the right-hand side is bounded above by the finite number $\sum^\infty_{j=0} \left(\frac{1}{2^{n-7}}\right)^{-j}$, which only depends on $n$; this completes the bound of the second term. So combining, we have
$$\H^{n+1}\left(\eball_{r/8}(\sing(V))\cap \eball_{1/2}\right) \leq C_0\folding(V) r^{n+1} + C^*_1\folding(V)r^8 + C_2 r^{8} \leq \tilde{C}(1+\folding(V))r^8$$
for some constants $C_0,C^*_1,C_2,\tilde{C}$ only depending on $n$ and $\Lambda$; here we have used that $r^{n+1}<r^8$; this therefore completes the proof of the first claimed bound.

To prove the second inequality, we shall use the first to prove a packing estimate. Indeed, we claim the following: for any $r\in (0,1/4]$, we can find a covering of $\sing(V)\cap B_{1/4}$ by balls $\{B_r(x_i)\}_{i=1}^N$ with $x_i\in\sing(V)$ and $N\leq C_3r^{-(n-7)}$, where $C_3 = C_3(n,\Lambda)$. Indeed, to see this simply choose $x_i\in \sing(V)\cap B_{1/4}$ a maximal collection of points such that $\{B_{r/2}(x_i)\}_{i}$ is a pairwise disjoint collection. Then, by construction we have $\sing(V)\cap B_{1/4} \subset \bigcup_{i=1}^N B_r(x_i)$, and
\begin{align*}
N r^{n-7} = \sum_i r^{n-7} & \leq \omega_{n+1}^{-1}2^{n+1}\cdot r^{-8}\sum_i \H^{n+1}(B_{r/2}(x_i))\\
&\leq \omega_{n+1}^{-1}2^{n+1}\cdot r^{-8} \H^{n+1}(B_{r/2}(\sing(V)\cap B_{1/4}))\\
& \leq \omega_{n+1}^{-1}2^{n+1}\cdot r^{-8}\cdot C(n,\Lambda)r^{8}\\
& \equiv C_3
\end{align*}
where $C_3 = \omega^{-1}_{n+1}2^{n+1}C$, where $C = C(n,\Lambda)$ is the constant from the first inequality we have already established; this proves the packing estimate claimed. Thus, using this cover we now have
\begin{align*}
\|V\|(B_{r/8}(\sing(V))\cap B_{1/8}) & \leq \|V\|(B_{r}(\sing(V)\cap B_{1/4}))\\
&\leq \|V\|\left(\bigcup_{i=1}^NB_{2r}(x_i)\right) \leq \sum_{i=1}^N\|V\|(B_{2r}(x_i))\\
& \hspace{8.5em}\leq \sum_{i=1}^N (2r)^n\|V\|(B_1(x_i)) \leq 2^n\Lambda r^n N \leq 2^n\Lambda C_3 r^7
\end{align*}
where in the fourth inequality we have used the monotonicity formula for stationary integral varifolds to get that, as $x_i\in \sing(V)\cap B_{1/4}$ and $0<2r<1$, $(2r)^{-n}\|V\|(B_{2r}(x_i)) \leq \|V\|(B_1(x_i))\ (\leq \Lambda)$. Covering $B_{1/2}$ by at most $C(n)$ balls of radius $1/8$, repeating the same argument in each ball, and summing the estimates completes the proof of the second estimate.
\end{proof}

We also remark the following theorem, which generalises \cite{naber-valtorta}*{Theorem 1.8}:

\begin{theorem}
Let $n\in\mathbb{Z}_{\geq 8}$ and $\Lambda>0$. Then, there exists a constant $C_0 = C_0(n,\Lambda)\in (0,\infty)$ such that for any $V\in \salpha$ and $r\in (0,1/2]$,
$$\H^{n+1}(\{x\in\spt\|V\|: |A_V(x)|>r^{-1}\}\cap \eball_{1/2})\leq C_0(1+\folding(V))r^8;$$
$$\H^{n+1}(\badreg_{r}(V)\cap \eball_{1/2}) \leq C_0(1+\folding(V))r^8;$$
$$\|V\|(\badreg_r(V)\cap \eball_{1/2}) \leq C_0(1+\folding(V))r^7.$$
\end{theorem}

\begin{proof}
    The second estimate follows from Corollary \ref{corollary of sheeting theorem} and Corollary \ref{rescaled quantitative estimates} in much the same way we have already seen in the proof of Theorem \ref{thm:A-2}. The first follows from the second along with the fact that $|A_V(x)|\leq \left[\regscale_V^Q(x)\right]^{-1}$ (for the appropriate choice of $Q = Q(x)$) and so this set is contained within the set of points where $\regscale_V^Q(x)< r$, i.e. $\badreg_r(V)$. The third inequality then follows from the second in the same way as the second inequality in Theorem \ref{thm:A-2} followed from the first in that setting.
\end{proof}

\section{Generalization to Riemannian Manifolds: Theorem \ref{thm:B1}}

In this section we will detail how one can modify the proof of Theorem \ref{thm:A} seen in the previous sections to prove Theorem \ref{thm:B1}. So fix $(N^{n+1},g)$ a smooth Riemannian manifold. Write $\exp_x$ for the exponential map of $N$ at $x\in N$ and $\inj(x)\in (0,\infty]$ for the injectivity radius of $N$ at $x$.

Let us first precisely define what we mean by the index of the regular part of a stationary integral varifold $V$ in $N$, adapting Definition \ref{definition bounded index} using \cite{wickstable}*{Section 18}. First, we define what it means for $V$ to have finite index on its regular part on a normal coordinate ball in $N$. So let $x\in \spt\|V\|$, and let $\mathcal{N}_{\rho}(x)$ be the normal coordinate ball of radius $\rho\in (0,\inj(x))$ around $x$, which we will assume also obeys $\dim_\H(\sing(V)\cap \mathcal{N}_\rho(x))\leq n-7$. Set $\tilde{V}:= \left(\exp^{-1}_x\right)_\#(V\restrictv \mathcal{N}_\rho(x))$, which is then an integral $n$-varifold on $B^{n+1}_\rho(0)\subset T_{x}N\cong \R^{n+1}$, which is stationary with respect to the functional
$$\mathcal{F}_{x}(\tilde{V}):= \int_{B^{n+1}_\rho(0)\times G(n,n+1)}|\Lambda_n D\exp_x(y)\circ S|\ d\tilde{V}(y,S).$$
For $\psi\in C^1_c(B^{n+1}_\rho(0)\backslash\sing(\tilde{V});\R^{n+1})$, the second variation with respect to $\mathcal{F}_x$ is then given by (see \cite{SS}*{(1.8), (1.10), (1.12)})
$$\delta^2_{\mathcal{F}_x}\tilde{V}(\psi) = \int_{\reg(\tilde{V})}\left\{\sum^n_{i=1}|(D_{\tau_i}\psi)^\perp|^2 + (\divergence_{\reg(\tilde{V})}\psi)^2 - \sum^n_{i,j=1}(\tau_i\cdot D_{\tau_j}\psi)\cdot(\tau_j\cdot D_{\tau_i}\psi)\right\}\ d\H^n + R(\psi)$$
where $\{\tau_1,\dotsc,\tau_n\}$ is an orthonormal basis for the tangent space $T_y(\reg(\tilde{V}))$ of $\reg(\tilde{V})$ at $y$, $D_\tau\psi$ denotes the directional derivative of $\psi$ in the direction $\tau$, and
$$|R(\psi)|\leq c\mu\int_{\reg(\tilde{V})}\left\{\tilde{c}\mu|\psi|^2 + |\psi||\nabla\psi| + |y||\nabla\psi|^2\right\}\ d\H^n(y)$$
where $c,\tilde{c}$ are absolute constants and $\mu$ is a constant depending only on the metric on $N$. Since $\reg(\tilde{V})$ is orientable on this ball (as the size of the singular set is sufficiently small and $\tilde{V}$ is codimension one), we may choose a continuous choice of unit normal $\nu$ to $\reg(\tilde{V})$ and, for any $\zeta\in C^1_c(\reg(\tilde{V}))$, extend $\zeta\nu$ to a vector field in $C^1_c(B^{n+1}_\rho(x)\backslash\sing(\tilde{V});\R^{n+1})$ and take in the above $\psi = \zeta\nu$ to deduce that
$$\delta^2_{\mathcal{F}_x}\tilde{V}(\zeta)\equiv\delta^2_{\mathcal{F}_x}\tilde{V}(\psi) = \int_{\reg(\tilde{V})}\left\{|\nabla\zeta|^2 - |A|^2\zeta^2 + H^2\zeta^2\right\}\ d\H^n + R(\psi)$$
where $A$ denotes the second fundamental form of $\reg(\tilde{V})$, $|A|$ the length of $A$, $H$ the mean curvature of $\reg(\tilde{V})$, and
$$|R(\psi)| \leq c\mu\int_{\reg(\tilde{V})}\left\{\tilde{c}\mu|\zeta|^2 + |\zeta||\nabla\zeta| + \zeta^2|A||y||\nabla\zeta|^2 + |y|\zeta^2|A|^2\right\}\ d\H^n(y).$$
We then say that the \textit{index of the regular part of $V$ in the normal coordinate ball $\mathcal{N}_\rho(x)$} (which we stress was assumed to obey $\dim_\H(\sing(V)\cap \mathcal{N}_\rho(x))\leq n-7$), denoted $\index(\reg(V);\mathcal{N}_\rho(x))$, is the dimension of the largest subspace $P$ of $\zeta\in C^1_c(\reg(\tilde{V}))$, where $\tilde{V}$ is as above, such that for all $\zeta\in P$,
$$\delta^2_{\mathcal{F}_x}\tilde{V}(\zeta) <0.$$
For a collection of disjoint normal coordinate balls $\mathcal{B}:= \{\mathcal{N}_{\rho_\beta}(x_\beta)\}_{\beta\in B}$ in $N$ such that $\dim_\H(\sing(V)\cap \mathcal{N}_{\rho_\beta}(x_\beta))\leq n-7$ for each $\beta\in B$, we say that the index of this collection, denoted $\index(\reg(V);\mathcal{B})$, is
$$\index(\reg(V);\mathcal{B}) := \sum_{\beta\in B}\index(\reg(V);\mathcal{N}_{\rho_\beta}(x_\beta)).$$
Finally, we say that the \textit{index of the regular part of} $V$, denoted $\index(\reg(V))$, is
$$\index(\reg(V)) := \sup_{\mathcal{B}}\index(\reg(V);\mathcal{B})$$
where this supremum is taken over all collections $\mathcal{B}$ is disjoint normal coordinate balls in $N$ with small singular set as above. We can then define the index of the regular part of $V$ in a subset $A\subset N$, denoted $\index(\reg(V);A)$, via the index of the regular part of $V\restrictv A$. We say that the regular part of $V$ is \textit{stable} if $\index(\reg(V)) = 0$.

One may then prove, in an analogous manner to that seen in Lemma \ref{locally stable}, that about each point $x\in \spt\|V||$, there is a radius $\rho_x\in (0,\inj(x))$ such that $\index(\reg(V);B^N_{\rho_x}(x)) = 0$, i.e. $V$ is stable in $B^N_{\rho_x}(x)$; here, $B^N_\rho(x)$ is the usual Riemannian ball in $N$ centred at $x$ of radius $\rho$, defined to be the set of points $y$ in $N$ such that the infimum of the length over all paths connecting $x$ to $y$ is $<\rho$. In particular, if we assume that $\H^{n-1}(\sing(V)\cap N) = 0$, one can then invoke the regularity theory of \cite{wickstable}*{Theorem 18.1} to see that necessarily $\dim_\H(\sing(V)\cap N)\leq n-7$, and so in fact the above definition of index includes each normal coordinate ball $B^N_\rho(x)$ in $N$.

We then define the \textit{stability radius} at each $x\in \spt\|V\|$ in the same manner as in Definition \ref{defn:stab-radius}, replacing Euclidean balls in the definition by the balls $B^N_\rho(x)$, i.e.
$$s_V(x):= \sup\{r\geq 0:\index(\reg(V);B^N_r(x)) = 0\}.$$
The discussion above tells us that $s_V(x)>0$ at every point $x\in \spt\|V\|$. However, it is only now clear that $s_V$ is \textit{locally} Lipschitz, i.e. for each $x\in \spt\|V\|$, there is a radius $\tilde{\rho}_x>0$ such that $s_V$ is Lipschitz (with Lipschitz constant at most $2$) on the ball $B^N_{\tilde{\rho}_x}(x)$; this follows in an analogous manner to Lemma \ref{stability radius continuity} on sufficiently small balls in $N$ for which $(N^{n+1},g)$ is close to the Euclidean ball of dimension $n+1$. In particular, $s_V$ is still a continuous function (when it is finite, i.e. when $V$ is unstable).

We define the \textit{folding number} in the same manner as Definition \ref{defn:folding}, except now we restrict only to open subsets which are given by normal coordinate balls (this restriction is still sufficient for our later purposes). Thus, $\folding(V)$ is defined to be the largest size of a (possibly infinite) collection of disjoint normal coordinate balls which are unstable. Analogously to Lemma \ref{bounded index disjoint sets}, we see that if $\index(\reg(V))\leq I<\infty$, then $\folding(V)\leq I$.

Now let us additionally assume that $0\in N$ and for some $K\in (0,\infty)$, we have $\left|\left.\textnormal{sec}\right|_{B^N_2(0)}\right| \leq K$ and $\left.\inj\right|_{B^N_2(0)}\geq K^{-1}$. Let us write $\salpha^*$ for the collection of all stationary integral varifolds $V$ in $B^N_2(0)\subset N$ which have $\H^{n-1}(\sing(V)\cap B^N_2(0)) = 0$ and whose regular part has finite index in the manner explained above. Let us also write $\tilde{\salpha}^*$ for the collection of all varifolds of the form
$$\tilde{V}:= (\eta_{0,\rho})_\#(\exp^{-1}_x)_\#(V\restrictv \mathcal{N}_\rho(x))$$
where $V\in \salpha^*$, $x\in \spt\|V\|$, and $\rho\in (0,\inj(x))$. We write $\salpha^*_I$ and $\tilde{\salpha}^*_I$ for the corresponding subsets of $\salpha^*$ and $\tilde{\salpha}^*$ respectively when the varifolds have regular part of index at most $I$.

For $\tilde{V}\in \tilde{\salpha}^*$, we may define the regularity scale $\regscale^Q_{\tilde{V}}$ analogously to that in Definition \ref{regularity scale}. We may use this to define the regularity scale of $V\in \salpha^*$ via
$$\regscale^{*Q}_{V}(x):= \sup\{\rho\in (0,\inj(x)): \regscale^Q_{\tilde{V}_\rho}(0) = 1\}$$
where $\tilde{V}_\rho := (\eta_{0,\rho})_\#(\exp^{-1}_x)_\#(V\restrictv \mathcal{N}_\rho(x))\in \tilde{\salpha}^*$ and again we set $\sup(\emptyset):= 0$. One may now follow the same proof as in Lemma \ref{regular implies positive regscale}, using instead the Riemannian versions of Wickramasekera's regularity theorem (i.e. \cite{wickstable}*{Theorem 18.2 and Theorem 18.3}), to prove that, for $V\in \salpha^*$, $x\in \reg(V)$ if and only if $\regscale^{*Q}_V(x)>0$ for some positive integer $Q$.

Wickramasekera's compactness (i.e. Theorem \ref{compactness theorem}) for stable varifolds still holds in the Riemannian setting, appropriately modified. For the readers convenience, we restate this result in the current setting:

\begin{theorem}[\cite{wickstable}*{Theorem 18.1}]
    Let $(N^{n+1},g)$ be a smooth Riemannian manifold and $x\in N$. Suppose that $(V_i)_i\subset \salpha^*_0$ is a sequence with $x\in \spt\|V_i\|$ for each $i=1,2,\dotsc$ and with $\limsup_{i\to\infty}\|V_i\|(N)<\infty$. Then, there exists a subsequence $(V_{i_j})_j$ and $V\in \salpha^*_0$ with $x\in \spt\|V\|$ such that $V_{i_j}\to V$ as varifolds in $N$ and smoothly (i.e. in the $C^k$ topology for every $k$) locally in $N\backslash\sing(V)$.
\end{theorem}

To complete the recasting of all results in Section \ref{sec:prelim} to the Riemannian setting, we finally have the following modified version of Theorem \ref{sheeting theorem}, which follows from \cite{wickstable}*{Theorem 18.2} in the same manner as before:

\begin{theorem}[Sheeting Theorem, \cite{wickstable}*{Theorem 18.2}]\label{thm:Riemann-sheeting}
    Let $n\geq 2$, $\Lambda>0$, and $K>0$. Let $(N^{n+1},g)$ be a smooth Riemannian manifold which obeys $0\in N$ and $\left.\inj\right|_{B^N_2(0)}\geq K^{-1}$. Then, there exists $\varepsilon_0 = \varepsilon_0(n,\Lambda,K)\in (0,1/4)$ and $Q_0 = Q_0(n,\Lambda,K)\in \mathbb{Z}_{\geq 1}$ such that the following is true: whenever $\tilde{V}\in \tilde{\salpha}^*_0$ satisfies:
    \begin{enumerate}
        \item [\textnormal{(a)}] $\omega_n^{-1}\|\tilde{V}\|(B^{n+1}_1(0)) \leq \Lambda$;
        \item [\textnormal{(b)}] $\sigma^{-1}\dist_\H(\spt\|\tilde{V}\|\cap (\R\times B_\sigma(0)), \{0\}\times B_\sigma)< \varepsilon_0$, for some $\sigma<\varepsilon_0$;
        \item [\textnormal{(c)}] $(\sigma^n\omega_n)^{-1}\|\tilde{V}\|(B^{n+1}_\sigma(0))\leq \Lambda$;
    \end{enumerate}
    then we have $\regscale^{Q_0}_{\tilde{V}}(0)\geq \sigma/2$.
\end{theorem}

Now let us turn our attention to Section \ref{sec:strata}. The notion of a varifold $V\in \salpha^*$ being $(\delta,r,k)$-conical at a point $x\in N$, as originally defined in Definition \ref{defn:conical}, is modified for the Riemannian setting using the corresponding varifold in the varifold class $\tilde{\salpha}^*$ as done in \cite{naber-valtorta}*{Definition 1.1(2)}; in particular, this is only defined when $r<\inj(x)$. This allows us to define the various strata as in Definition \ref{defn:strata} in the same manner. One must be careful as we no longer necessarily have a suitable homothety map $\eta_{x,\rho}$ (as this would modify $N$ also), so Lemma \ref{rescaled strata} must be understood differently and instead as applying only to varifolds in $\tilde{\salpha}^*$. Using Theorem \ref{thm:Riemann-sheeting}, we may prove the following Riemannian variant of Theorem \ref{epsilon regularity}:

\begin{theorem}[Riemannian $\varepsilon$-Regularity Theorem]\label{thm:Riemann-epsilon-reg}
    Let $n\geq 2$ and $\Lambda,K,d\in (0,\infty)$. Let $(N^{n+1},g)$ be a smooth Riemannian manifold which obeys $0\in N$ and $\left.\inj\right|_{B^N_2(0)}\geq K^{-1}$. Let $A\subset B^N_2(0)$ be a compact subset which obeys $d(A,\del B^N_2(0))\geq d$. Then, there exist constants $\varepsilon_0 = \varepsilon_0(n,\Lambda,K,d)$ and $Q_0 = Q_0(n,\Lambda,K,d)\in \mathbb{Z}_{\geq 1}$ such that the following holds: if $V\in \salpha^*$, $x\in \spt\|V\|\cap A$, and $\rho\in (0,d]$ satisfy:
    \begin{enumerate}
        \item [\textnormal{(a)}] $\|V\|(B^{N}_2(0))\leq \Lambda$;
        \item $V$ is stable in $B^N_{\rho/2}(x)$;
        \item $V$ is $(\varepsilon_0,\rho/2,n-6)$-conical at $x$;
    \end{enumerate}
    then we have $\regscale^{*Q_0}_V(x)\geq \varepsilon_0\rho$.
\end{theorem}

This is established in the same manner as Theorem \ref{epsilon regularity}; note that we are again applying the compactness theorem to the corresponding varifolds in $\tilde{\salpha}^*$, which does hold (either by \cite{wickstable}*{Section 18} or also \cite{SS}*{Theorem 2} as our varifolds have sufficiently small singular set here) and then applying Theorem \ref{thm:Riemann-sheeting} with $\sigma = \varepsilon_1/2$, for the choice of constant $\varepsilon_1 = \varepsilon_1(n,\Lambda,K)\in (0,1/4)$ in Theorem \ref{thm:Riemann-sheeting}.

We then have from Theorem \ref{thm:Riemann-epsilon-reg} the corresponding Riemannian version of Corollary \ref{corollary of sheeting theorem}:

\begin{corollary}\label{cor:Riemann-version}
    Let $n\geq 7$ and $\Lambda,K,d\in (0,\infty)$. Let $(N^{n+1},g)$ be a smooth Riemannian manifold which obeys $0\in N$ and $\left.\inj\right|_{B^N_2(0)}\geq K^{-1}$. Let $A\subset B^N_2(0)$ be a compact subset which obeys $d(A,\del B^N_2(0))\geq d$. Then, there exist constants $\varepsilon_0 = \varepsilon_0(n,\Lambda,K,d)$ and $Q_0 = Q_0(n,\Lambda,K,d)\in \mathbb{Z}_{\geq 1}$ such that the following holds: if $V\in \salpha^*$ obeys:
    \begin{enumerate}
        \item [\textnormal{(a)}] $\|V\|(B^N_2(0))\leq \Lambda$;
        \item [\textnormal{(b)}] $V$ is stable in $B^N_d(x)$ for all $x\in A$;
    \end{enumerate}
    then we have $\badreg^{Q_0}_{\varepsilon_0\sigma}(V)\cap A\subset \strata^{n-7}_{\varepsilon_0,\sigma/2,d}(V)\cap A$ for all $\sigma\in (0,d]$.
\end{corollary}

Finally, we detail how one can modify the results of Section \ref{sec:estimates} to the Riemannian setting given the above. The first point to note is that the local estimates of Naber--Valtorta hold in the Riemannian setting as long as one assumes a lower bound on the injectivity radius as well as an absolute bound on the sectional curvature:

\begin{theorem}[\cite{naber-valtorta}*{Theorem 1.3}]
    Let $n\geq 2$ and $\Lambda,K,H\in (0,\infty)$, and $\varepsilon\in (0,1)$. Then, there exists a constant $C_\varepsilon = C_\varepsilon(n,\Lambda,K,\varepsilon)\in (0,\infty)$ such that the following is true: let $(N^{n+1},g)$ be a smooth Riemannian manifold obeying $0\in N$, $\left|\left.\textnormal{sec}\right|_{B^N_2(0)}\right| \leq K$ and $\left.\inj\right|_{B^N_2(0)}\geq K^{-1}$. Suppose that $V$ is a stationary integral $n$-varifold in $B^N_2(0)$ obeying $\|V\|(B^{N}_2(0))\leq\Lambda$ which has first variation bounded by $H$. Then, for each $k\in \{0,1,\dotsc,n\}$ we have
    $$\H^{n+1}\left(B^N_r(\strata^k_{\varepsilon,r,1}(V))\cap B^N_1(0)\right) \leq C_\varepsilon r^{n+1-k}\ \ \ \ \text{for all }r\in (0,1].$$
\end{theorem}

The rescaled version of this corresponding to Corollary \ref{rescaled quantitative estimates} also holds for radii $R$ at most the injectivity radius at the point by passing through to the varifolds $\tilde{\salpha}^*$ via the exponential map and applying the rescaled estimates there (note that if we rescale $N$, one can still bound the scalar curvature and injectivity radius of the rescaled version of $N$ by $K$) and then pushing these back to the Riemannian manifold level; note that this will introduce some additional constant which depends on the metric $g$ (through the exponential map) but this can be controlled in terms of $K$, so the new constant has the same dependencies.

Thus, in order to rerun the proofs of Lemma \ref{low stability bound}, Lemma \ref{covering lemma}, and Theorem \ref{thm:A-2} using the above modified results for the Riemannian setting (as the proofs will be unchanged up to constant factors) we need a suitable version of the Besicovitch covering lemma for our setting. However, Riemannian manifolds with a lower bound on the sectional curvature and have finite diameter (as is our situation here when working on $B^N_2(0)$) necessarily are \textit{directionally limited}, as defined in \cite{federer}*{Definition 2.8.9}. As such, a form of the Besicovitch covering theorem does hold in this setting, by \cite{federer}*{Theorem 2.8.14}, and moreover the corresponding Besicovitch constant only depends on $n$ and $K$. Hence, our arguments in Section \ref{sec:estimates} pass through to the current setting, and prove Theorem \ref{thm:B1}.

To see Theorem \ref{thm:B}, note that closed Riemannian manifolds have the property that there is a $K = K(N,g)\in (0,\infty)$ such that $|\textnormal{sec}|\leq K$ and $\inj\geq K^{-1}$ (and indeed have finite diameter, so a Besicovitch theorem will hold on all of $N$), and as such the result follows from Theorem \ref{thm:B1} by a simple covering and compactness argument.

\bibliographystyle{alpha}
\bibliography{minkowski-ref-arxiv}

\end{document}